\documentclass[article,12pt]{aiaa-pretty}    % journal, submit, article
\usepackage{amssymb}
\usepackage{amsmath}
\usepackage{booktabs}
 \usepackage{varioref}%  smart page, figure, table, and equation referencing
 \usepackage{wrapfig}%   wrap figures/tables in text (i.e., Di Vinci style)
 \usepackage{threeparttable}% tables with footnotes
 \usepackage{dcolumn}%   decimal-aligned tabular math columns
  \newcolumntype{d}{D{.}{.}{-1}}
 \usepackage{nomencl}%   nomenclature generation via makeindex
  \makeglossary
 \usepackage{subfigure}% subcaptions for subfigures
 \usepackage{graphicx}
 \usepackage{subfigmat}% matrices of similar subfigures, aka small mulitples
 \usepackage{fancyvrb}%  extended verbatim environments
  \fvset{fontsize=\footnotesize,xleftmargin=2em}
 \usepackage{lettrine}%  dropped capital letter at beginning of paragraph
\newtheorem{theorem}{Theorem}
\newtheorem{proposition}{Proposition}
\newtheorem{remark}{Remark}
\newenvironment{proof}{{\it Proof. }}{\hfill $\Box$}

\newcommand{\Real}{\mathbb R}
\newcommand{\set}[1]{\left\{#1\right\}}
\newcommand{\real}[1]{{\mathbb R}^{#1}}
\newcommand{\bb}{{\boldsymbol b}}

\newcommand{\bC}{{\boldsymbol C}}

\newcommand{\be}{{\boldsymbol e}}
\newcommand{\bff}{{\boldsymbol f}}

\newcommand{\bh}{{\boldsymbol h}}

\newcommand{\bl}{{\boldsymbol l}}

\newcommand{\bu}{{\boldsymbol u}}

\newcommand{\bx}{\boldsymbol x}

\newcommand{\bxf}{{\bx(\cdot)}}  % f for function
\newcommand{\buf}{{\bu(\cdot)}}  % f for function
\newcommand{\bA}{{\boldsymbol A}}
\newcommand{\bB}{{\boldsymbol B}}
\newcommand{\bD}{{\boldsymbol D}}

\newcommand{\bI}{{\boldsymbol I}}
\newcommand{\bM}{{\boldsymbol M}}

\newcommand{\X}{\mathbb{X}}
\newcommand{\U}{\mathbb{U}}

\newcommand{\bzero}{{\bf 0}}

\author{ %
N. Koeppen\thanks{Graduate Student and Captain, USMC, Control and Optimization Laboratories, Department of Mechanical and Aerospace Engineering },
I. M. Ross\thanks{Distinguished Professor and Program Director, Control and Optimization, Department of Mechanical and Aerospace Engineering},
L. C. Wilcox\thanks{Associate Professor, Department of Applied Mathematics},
R. J. Proulx\thanks{Research Professor, Control and Optimization Laboratories, Space Systems Academic Group}\\
\textit{Naval Postgraduate School, Monterey, CA 93943}
}
\title{Fast Mesh Refinement in Pseudospectral Optimal Control}

\abstract{
Mesh refinement in pseudospectral (PS) optimal control is embarrassingly easy --- simply increase the order $N$ of the Lagrange interpolating polynomial and the mathematics of convergence automates the distribution of the grid points. Unfortunately, as $N$ increases, the condition number of the resulting linear algebra increases as $N^2$; hence, spectral efficiency and accuracy are lost in practice.
In this paper, we advance Birkhoff interpolation concepts over an arbitrary grid to generate well-conditioned PS optimal control discretizations. We show that the condition number increases only as $\sqrt{N}$ in general, but is independent of $N$ for the special case of one of the boundary points being fixed. Hence, spectral accuracy and efficiency are maintained as $N$ increases.  The effectiveness of the resulting fast mesh refinement strategy is demonstrated by using \underline{polynomials of over a thousandth order} to solve a low-thrust, long-duration orbit transfer problem.

}

\begin{document}
\maketitle

%\end{spacing}

%\newpage
\section{Introduction}
In principle, mesh refinement in pseudospectral (PS) optimal control\cite{PSReview-ARC-2012} is embarrassingly easy --- simply increase the order $N$ of the Lagrange interpolating polynomial.
For instance, in a Chebyshev PS method\cite{boyd,fahroo:cheb-jgcd,cheb-costate}, the Chebyshev-Gauss-Lobatto (CGL) mesh points are given by\cite{boyd},
\begin{equation}\label{eq:CGL-t-domain}
t_i = \frac{1}{2} \left[(t_f + t_0) - (t_f -t_0)\cos\left(\frac{i\pi}{N}\right) \right], \quad i = 0, 1, \ldots, N
\end{equation}
where, $[t_0, t_f]$ is the time interval.  Grid points generated by \eqref{eq:CGL-t-domain} for $N = 5, 10$ and $20$ over a canonical time-interval of $[0, 1]$ are shown in Fig.~\ref{fig:CGLgrid}.
%
%======================================================================================
   \begin{figure}[h!]
      \centering
      {\includegraphics[angle=0, width = 0.6\textwidth]{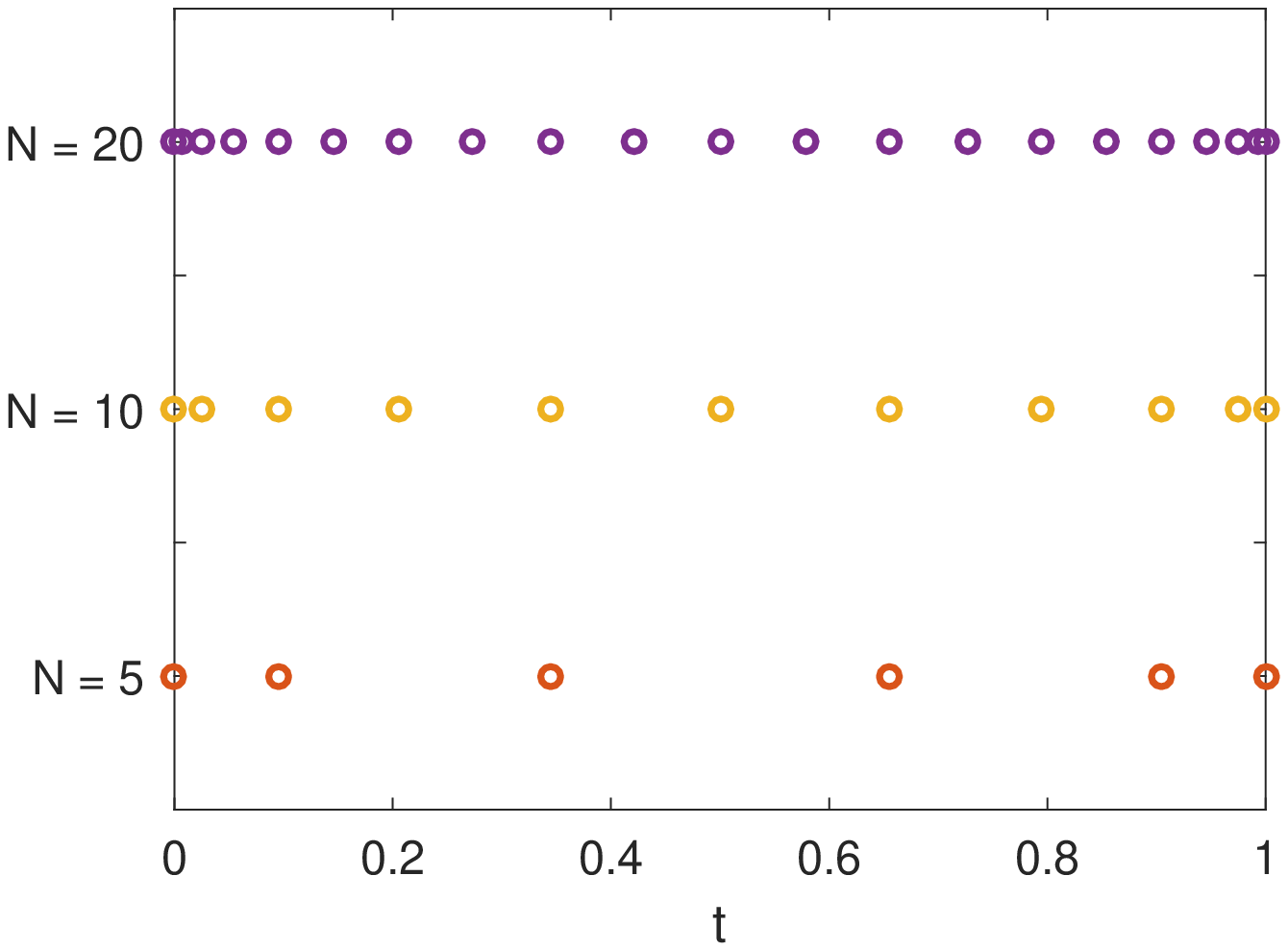}}
      \caption{\textsf{Grid points for $N = 5, 10$ and $20$ for a Chebyshev pseudospectral method.}}
    \label{fig:CGLgrid}
   \end{figure}
%==========================================================================================
%
The production of these mesh points is as fast as it takes to compute the cosine function. Furthermore, the mesh points are at optimal locations for a Chebyshev PS method because their placement is dictated by well-established rules of approximation theory\cite{atap}.
The price for this speed and simplicity is an increase in the condition number of the underlying linear algebra that governs the iterations\cite{PSReview-ARC-2012, boyd, shen-book}. In fact, as shown in Fig.~\ref{fig:CondNo},
%
%======================================================================================
   \begin{figure}[h!]
      \centering
      {\includegraphics[angle=0, width = 0.6\textwidth]{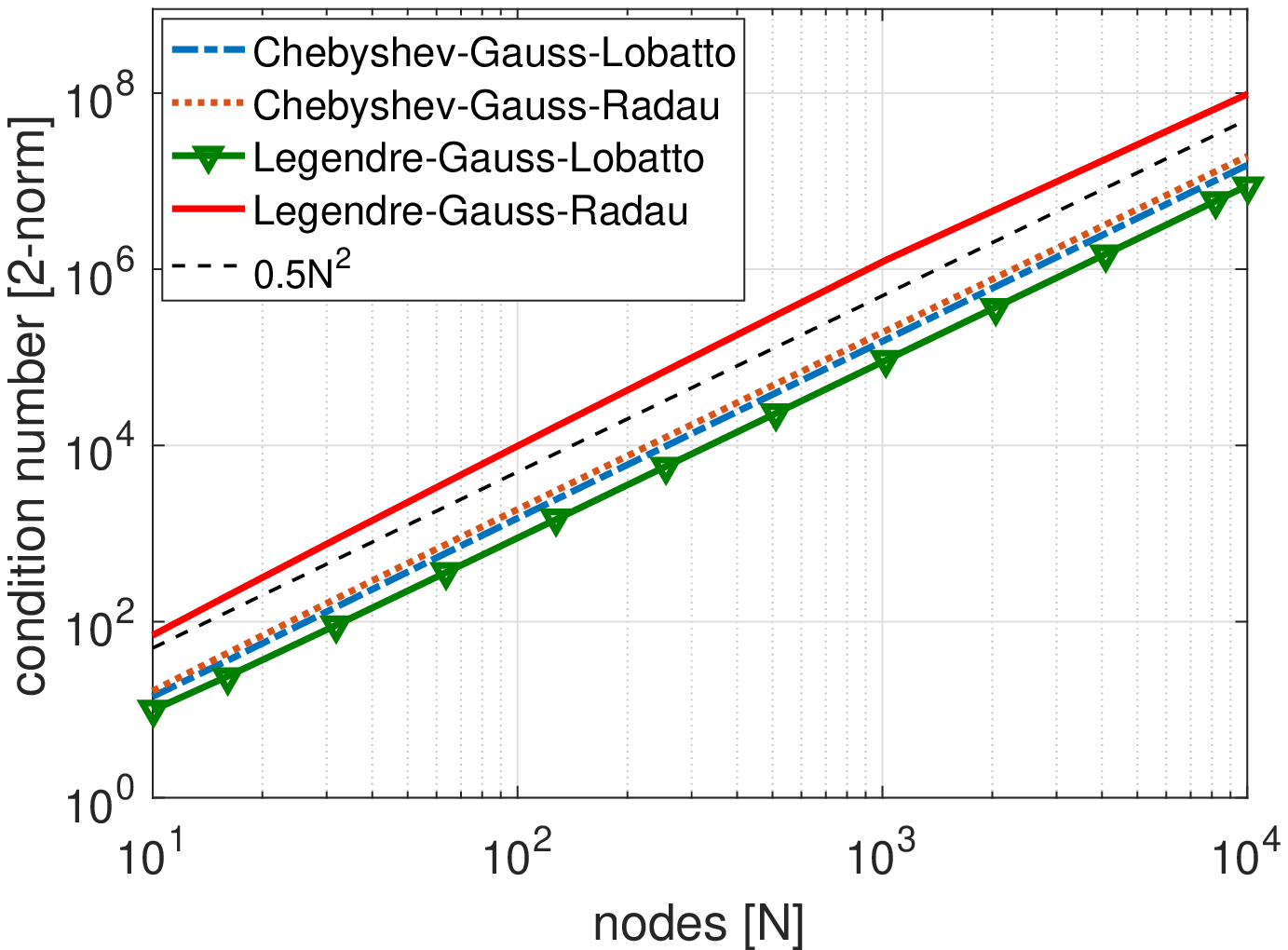}}
      \caption{\textsf{Illustrating $\mathcal{O}(N^2)$-growth in the condition number associated with \eqref{eq:xdot=f} for various mesh selections.}}
    \label{fig:CondNo}
   \end{figure}
%==========================================================================================
%
the condition number associated with the PS discretization of the dynamical equations,
\begin{equation}\label{eq:xdot=f}
\dot\bx(t) = \bff(\bx(t), \bu(t), t)
\end{equation}
increases as $N^2$.  The plots in this figure are the condition numbers of the non-singular, inner, square-portion of the various differentiation matrices\cite{PSReview-ARC-2012,boyd} associated with the indicated mesh selections. \emph{\textbf{Note that the Legendre-Gauss-Radau mesh has the worst performance of all the indicated PS grid selections.}}  An intuitive approach to manage the growth in the condition number is to break up the time interval into several segments and apply a PS method over each segment\cite{fahrooRoss:patching,knots,acc:hybrid,auto-knots}.  When the segments are non-overlapping and information across the segments is passed across a single point, it results in a PS knotting method\cite{knots,auto-knots}; see Fig.~\ref{fig:PSKnots}.
%
%======================================================================================
   \begin{figure}[h!]
      \centering
      {\includegraphics[angle=0, height=2.5in, width = 0.5\textwidth]{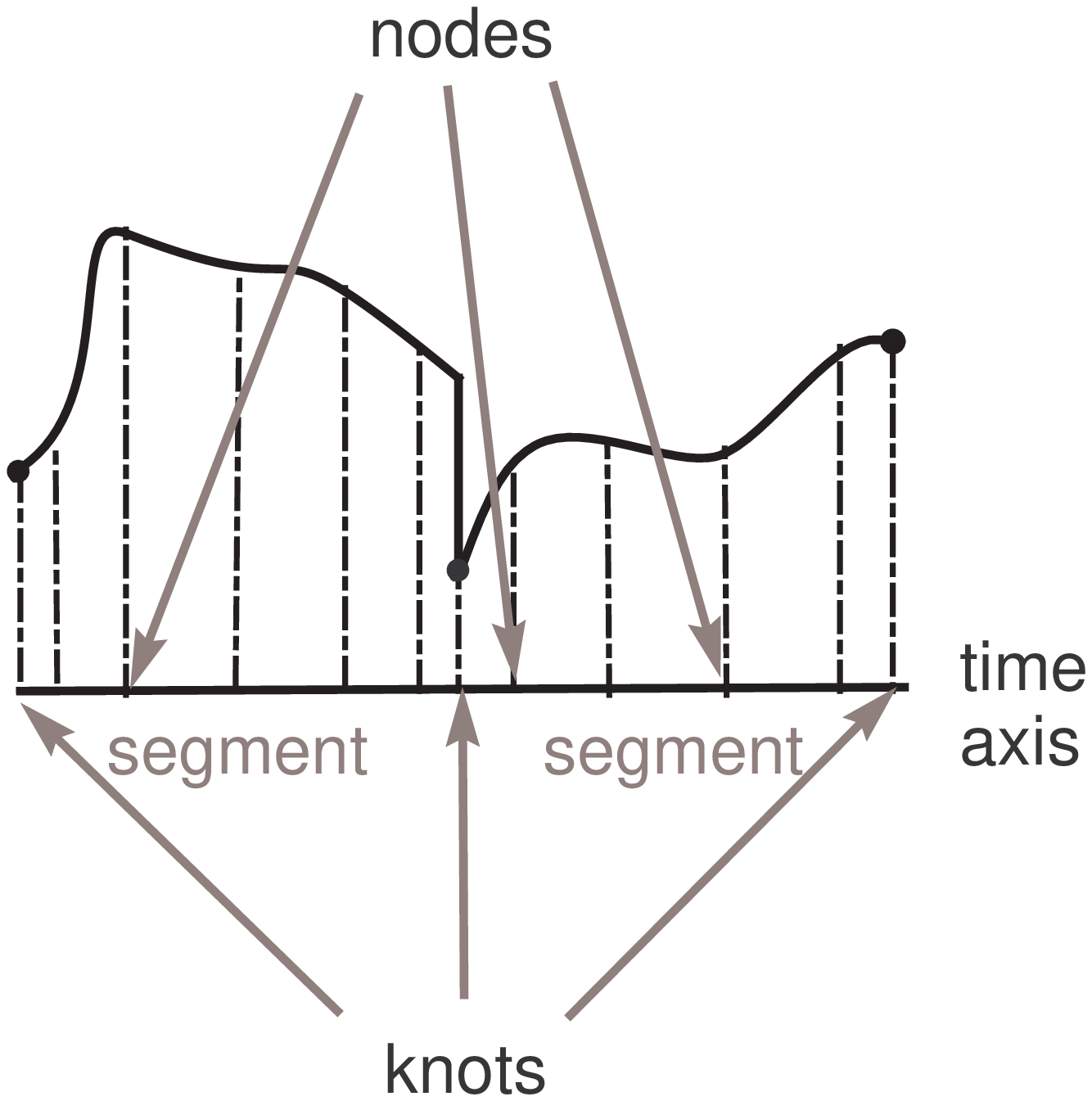}}
      \caption{\textsf{An illustration of PS knots in a generic pseudospectral knotting method; adapted from \cite{knots}.}}
    \label{fig:PSKnots}
   \end{figure}
%==========================================================================================
%
The special nodes where information across segments is transmitted are called PS knots.  In the spectral algorithm for PS methods\cite{auto-knots,spec-alg}, the number and location of the PS knots are determined adaptively. The adaptation process is based on an analysis of the solution generated by the prior iterate.  Among other things, the analysis involves an estimation of the error and a determination of the frequency content of the system trajectory.  Based on this analysis, either $N$ per segment is increased or new knots are placed near high frequency regimes\cite{auto-knots,spec-alg}. The entire process is repeated for the next iterate or until all stopping criteria are met\cite{PSReview-ARC-2012,knots,auto-knots,spec-alg}.

Although the PS knotting method manages the growth in the condition number by simply limiting the order of the approximating polynomial in a given segment, the price for this approach is quite severe.  First, there is an increase in the computation time to determine the number and location of the knots; second, there is a significant reduction in the rate of convergence\cite{TAC:linearizable,kang-ijrc,Kang_2008_convergence,kang-rate} that adds to the increase in computational time and effort.  The convergence rate varies inversely to the number of segments --- the more the segments, the poorer the convergence rate. If the order of the approximating polynomial is limited to low orders (e.g., $N=4$), a PS  method becomes theoretically equivalent to a Runge-Kutta method\cite{conway-GL,hnw-ode,conway:survey}.  In fact, the ``exponential'' rate of convergence of an unknotted PS method is chiefly due to the high orders (i.e., $N\gg10$) of the approximating polynomials\cite{boyd,atap}.  In many aerospace applications\cite{Karp-JWST,Kepler-micro-slew,RobStevens-Tether, CDC-Workshop},  including flight implementations\cite{PSReview-ARC-2012,zpm:IEEE,bhatt:opm,TRACE-IEEE-Spectrum,TRACE-JGCD-2014,Minelli-AAS,zpm:NASA-report}, very accurate solutions\cite{TEI-JGCD-2011,scaling} are generated for $N \le 100$. In certain emerging areas of applications\cite{ross-challenge,nolcos16,napa16}, there have been new requirements to generate solutions for $N\gg 100$.   Because the unknotted PS method generates high condition numbers, various types of preconditioning are often used to alleviate the problem\cite{hesthavan,elbarbary}. In recent years, there have been major advancements in well-conditioned PS methods\cite{wang,olver,du} for ordinary differential equations.  \emph{\textbf{These methods offer a remarkable drop in the condition number from $ \mathcal{O} (N^2)$ to $ \mathcal{O}(1)$. }}In this paper, we adapt these advancements for PS optimal control using Birkhoff interpolants\cite{wang}.
The dramatic drop in the condition number allows us to generate solutions for thousands of grid points (i.e., using approximating polynomials for state trajectories in excess of a thousand degrees) thereby generating a very fast mesh refinement strategy.  Numerical results for a sample low thrust trajectory optimization problem are presented to illustrate the advancements.

\vskip 0.4in
\fbox{
\begin{minipage}{0.9\textwidth}
The citation for the journal version of this paper is: \textit{J. Guid., Contr. \& Dyn., 42/4, 2019, pp.~711--722}.
\end{minipage}
}

\section{A Brief Review of the State-of-the-Art in PS Optimal Control Methods}
In order to provide a proper context for the issues in PS optimal control methods, we first briefly review the state of the art.  To this end, consider a generic optimal control problem given by\cite{ross-book, longuski,clarke-2013book},
%
%==========================================================
\newlength{\leftside}
\newlength{\rightside}
\newcommand*{\leftterm}{}
\newcommand*{\rightterm}{}
\newcommand*{\term}[1]{$\displaystyle#1$}
%============================================================================
\begin{align*}
%---------------------------------------------------
&\left.\begin{aligned}
\phantom{preamble\,\,p}
\X & = \real{N_x}   &\U & = \real{N_u}& \\
\bx &= (x_1, \ldots, x_{N_x})  \quad &\bu &= (u_1, \ldots, u_{N_u}) & \\
\end{aligned}\hspace{2.85cm} \right\} \ \text{(preamble)}
%---------------------------------------------------
\\
%\]
%\vspace{-18pt} % USE THIS TO ALIGN!!
%%==========================================================
%\newlength{\leftside}
%\newlength{\rightside}
%\newcommand*{\leftterm}{}
%\newcommand*{\rightterm}{}
%\newcommand*{\term}[1]{$\displaystyle#1$}
%-----------------------------------------------------------
%\[
%
&\begin{aligned}
\renewcommand*{\leftterm}{\phantom{\quad J[\bx(\cdot), \bu(\cdot), t_0, t_f]}}
\renewcommand*{\rightterm}{\phantom{\be(\bx_0, \bx_f, t_0, t_f) \le \be^U dt }}
\settowidth{\leftside}{\term{\leftterm}}
\settowidth{\rightside}{\term{\rightterm}}
%----------------------------------------------------------
%%---------------------------------------------------
%\begin{array}{rlrlrl}
%\X &=\real{3}   &\U &= \Real& \\
%\bx &= (x, y, v)  \quad &\bu &= \theta & \\
%\end{array}\\
%%---------------------------------------------------
\overbrace{(G)}^{\text{\normalsize problem}} \left\{
\begin{array}{ll}
\text{Minimize }& \\[-1.25em] % lines up Minimize with J
& \left.\begin{aligned}
\makebox[\leftside][r]{\term{J[\bx(\cdot), \bu(\cdot), t_0, t_f] := }} & \\
% \makebox[\leftside][r]{\term{= E(\bx_0, \bx_f, t_0, t_f)}} & \\ %+ \makebox[\rightside][l]{\term{Fsf}}\\
 \makebox[\leftside][r]{\term{E(\bx_0, \bx_f, t_0, t_f)}} & + \makebox[\rightside][l]{\term{\int_{t_0}^{t_f}F(\bx(t), \bu(t), t)\ dt }}  \\
 \end{aligned}\right\} \ \text{(cost)}\\ [2em]
\text{Subject to} &  %\begin{aligned}
%    \makebox[\leftside][r]{\term{\dot x}} &= \makebox[\rightside][l]{\term{v\sin\theta}}\\
%  \end{aligned}
\\[-1.5em]
&  \left.\begin{aligned}
    \makebox[\leftside][r]{\term{\dot\bx}} &= \makebox[\rightside][l]{\term{\bff(\bx(t), \bu(t), t)}}\\
  \end{aligned}\right\}\ \text{(dynamics)} \\ [1em]
  %
  %%
%&  \left.\begin{aligned}
%    \makebox[\leftside][r]{\term{\dot x}} &= \makebox[\rightside][l]{\term{v\sin\theta}} \\
%    \dot y &= v\cos\theta \\
%    \dot v &= g\cos\theta
%  \end{aligned}\right\} \quad \textsf{(dynamics)} \\[2em]
%  %
&  \left.\begin{aligned}
    \makebox[\leftside][r]{\term{\be^L}}&\le \makebox[\rightside][l]{\term{\be(\bx_0, \bx_f, t_0, t_f) \le \be^U}}\\
%    \makebox[\leftside][r]{\term{\be^L \le \be(\bx_0, \bx_f, t_0, t_f)}}&\le \makebox[\rightside][l]{\term{\be^U}}\\
%    \be^L &=  (0, 0)
  \end{aligned}\right\} \ \text{(events)} \\[1em]
  &  \left.\begin{aligned}
%    (t_0, x_0, y_0, v_0) &=  \makebox[\rightside][l]{\term{(0, 0, 0, 0)}}\\
   \makebox[\leftside][r]{\term{\bh^L}}
   &\le \makebox[\rightside][l]{\term{\bh(\bx(t), \bu(t), t) \le \bh^U}}\\
  \end{aligned}\right\} \ \text{(path)}
\end{array}
\right.
\end{aligned}\label{prob:B}
\end{align*}
%==================================================================
%
where, the symbols denote the following:
\begin{itemize}
\item $\X$  and $\U$ are $N_x$- and $N_u$-dimensional real-valued state and control spaces respectively. We assume $N_x \in \mathbb{N}^+ $ and $N_u \in \mathbb{N}^+ $.
\item $J$ is the scalar cost function.  The arguments of $J$ are the optimization variables.
\item The optimization variables are:
\begin{itemize}
\item $\bxf$: the $N_x$-dimensional state trajectory,
\item $\buf$: the $N_u$-dimensional control trajectory,
\item $t_0$: the initial clock time, and
\item $t_f$: the final clock time.
\end{itemize}
\item $E$ is the scalar endpoint cost function.  The arguments of $E$ are the endpoints.  In the classical literature, $E$ is known as the ``Mayer'' cost function.
\item The endpoints are the initial state $\bx_0 \equiv \bx(t_0)$, the final state $\bx_f \equiv \bx(t_f)$, the initial time $t_0$ and the final time $t_f$.
\item $F$ is the scalar running cost function.  The arguments of $F$ are the instantaneous value of the state variable $\bx(t)$, the instantaneous value of the control variable $\bu(t)$ and time $t$.
\item $\bff$ is the $N_x$-dimensional ``dynamics function,'' or more appropriately the right-hand-side of the dynamics equation. The arguments of $\bff$ are exactly the same as the arguments of $F$.
\item $\be$ is the $N_e$-dimensional endpoint constraint function.  The arguments of $\be$ are exactly the same as that of $E$.
\item $\be^L$ and $\be^U$ are the $N_e$-dimensional lower and upper bounds on the values of $\be$.
\item $\bh$ is the $N_h$-dimensional path constraint function. The arguments of $\bh$ are exactly the same as that of $F$.
\item $\bh^L$ and $\bh^U$ are the $N_h$-dimensional lower and upper bounds on the values of $\bh$.
\end{itemize}
%
%======================
%
As shown in \cite{advances}, many key aspects of PS optimal control can be understood by considering a distilled optimal control problem given by,
\begin{eqnarray}
&x \in \Real, \quad u \in \Real, \quad \tau \in [-1, 1]   & \nonumber\\
& (\textsf{$P$}) \left\{
\begin{array}{lrl}
\emph{Minimize } & J[x(\cdot), u(\cdot)] =& E(x(-1),x(1))\\
\emph{Subject to}& \dot x(\tau) =& f(x(\tau), u(\tau))  \\
& e(x(-1), x(1))  = & 0
\end{array} \right. & \label{eq:probP}
\end{eqnarray}
The PS results generated for Problem $P$ easily generalize to Problem $G$, albeit with some tedious bookkeeping.  In order to avoid the distractions of bookkeeping, we limit many of our subsequent discussions to Problem $P$ while noting that all of the results -- with appropriate reindexing -- apply to Problem $G$ as well.  See also \cite{PSReview-ARC-2012} for a more thorough review and survey of the key results.

In PS optimal control theory, the state trajectory $x(\cdot)$ is expressed -- in principle -- as a generalized Fourier expansion,
\begin{equation}\label{eq:modal}
x(\tau) = x^\infty(\tau) := \lim_{N \to \infty} x^N(\tau) :=\lim_{N \to \infty} \sum_{m=0}^N a_m P_m(\tau)
\end{equation}
where, $a_m \in \Real$ is an unknown ``spectral'' coefficient of an $m^{th}$-order polynomial $P_m$.
Note that \eqref{eq:modal} is an exact representation and not an approximation.  The justification for \eqref{eq:modal} comes from the classic Stone-Weierstrass theorem\cite{Tao:epsilon} which guarantees a polynomial representation for $x(\cdot)$ under mild assumptions.  For most aerospace applications, the ``best'' choices for $P_m$ are  the Legendre and Chebyshev polynomials; hence, PS optimal control techniques resulting from these ``big two'' polynomials are known as the Legendre and Chebyshev PS methods respectively. A key principle in a generic PS method is that the spectral coefficients $a_m$ are computed indirectly by transforming \eqref{eq:modal} to an equivalent ``pseudospectral'' form\cite{ross-book, arb-grid-aas, arb-grid},
\begin{equation}\label{eq:nodal}
x^N(\tau) := \sum_{j=0}^N x^N(\tau_j)L_j(\tau)
\end{equation}
where, $\tau_j, j = 0, 1, 2, \ldots, N $ are discrete points in time, known as nodes, and $L_j$ are Lagrange interpolating polynomials.  That is, $L_j(\tau), j = 0, 1, 2, \ldots, N $ are such that they  satisfy the Kronecker relationship,
\begin{equation}\label{eq:lag=kron}
 L_j(\tau_i) = \delta_{ij}
\end{equation}
\emph{\textbf{Equations \eqref{eq:modal} and \eqref{eq:nodal}  are known as the modal and nodal representations of $x(\cdot)$ respectively\cite{ross-book}}}.

Let,
\begin{equation}
\pi^N := [\tau_0, \tau_1, \ldots, \tau_N]
\end{equation}
be an arbitrary grid (or ``mesh'') of points (see Fig.~\ref{fig:ArbGrid}) such that
\begin{equation}\label{eq:arb-grid}
-1 \le \tau_0 < \tau_1 < \cdots < \tau_{N-1} < \tau_N \le 1
\end{equation}
Note that \eqref{eq:arb-grid} allows the extreme points of the grid, $\tau_0$ and $\tau_N$, to be at or near the endpoints.  A nodal representation of the state trajectory over $\pi^N$ is given by,
\begin{equation}\label{eq:Lag-x}
 x^N(\tau) := \sum_{j=0}^N x_j L_j(\tau)
\end{equation}
%
%
%======================================================================================
   \begin{figure}[h!]
      \centering
      {\includegraphics[angle=0, width = 0.5\textwidth]{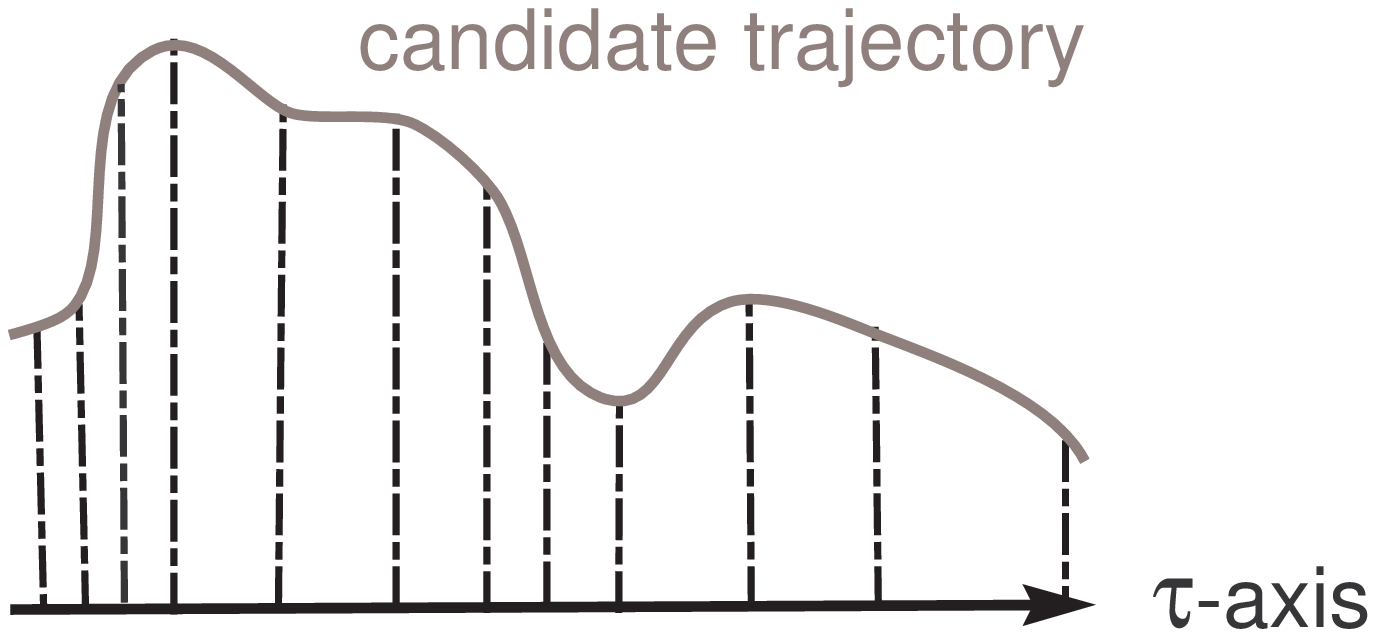}}
      \caption{\textsf{Illustration of an arbitrary grid in the generic pseudospectral method developed in \cite{arb-grid-aas} and \cite{arb-grid}.}}
    \label{fig:ArbGrid}
   \end{figure}
%==========================================================================================
%
Applying \eqref{eq:lag=kron} to \eqref{eq:Lag-x}, it follows that $x_i = x^N(\tau_i), \ i = 0, \ldots, N$.
Differentiating \eqref{eq:Lag-x} and evaluating $x^N(t)$ over the grid we get,
\begin{equation}\label{eq:xNdot}
\dot x^N(\tau_i) = \sum_{j=0}^N x_j \dot L_j(\tau_i)
\end{equation}
Define
\begin{equation}\label{eq:D:=}
\bD := \Big(\dot L_j(\tau_i)  \Big)_{0\le i,j \le N}
\end{equation}
as the PS differentiation matrix, and  let,
\begin{equation}\label{eq:X:=}
X := [x_0, x_1, \ldots, x_N]^T
\end{equation}
Then, using \eqref{eq:D:=} and \eqref{eq:X:=},  \eqref{eq:xNdot} can be written as a matrix-vector product given by,
\begin{equation}\label{eq:XDOT=DX}
\dot X = \bD X
\end{equation}
\emph{\textbf{Note that in sharp contrast to Runge-Kutta methods, PS methods rely almost entirely on an accurate discretization of $\dot x$ rather than $f(x, u)$, the right-hand-side of the differential equation,~\eqref{eq:xdot=f}}}.

A nodal representation of the control trajectory over $\pi^N$ is given by,
\begin{equation}\label{eq:uN}
u^N(\tau) = \sum_{j=0}^N u_j\, C_j(\tau)
\end{equation}
where, $C_j(\tau)$ is some interpolating function. 
%
%-----------------------------
\begin{remark}
Note that the control trajectory $\tau \mapsto u^N(\tau)$ is not necessarily a polynomial.\footnote{Unfortunately, it is widely and erroneously reported in the optimal control literature that $C_j$ is a polynomial.}  That is, $C_j(\tau)$ is any function that satisfies,
$$ C_j(\tau_i) = \delta_{ij}  $$
The non-polynomial form of $t \mapsto u^N(t)$ is exploited in \cite{TAC:linearizable,kang-ijrc,Kang_2008_convergence,kang-rate} for a strong proof of convergence; that is, a convergence proof that does not rely on unnecessary assumptions\cite{hager} that are typically not satisfied in many practical optimal control problems.  Note as well that the Bellman pseudospectral method\cite{bellman-low-t,bellman-conf} also capitalizes on the non-polynomial form of $t \mapsto u^N(t)$.
\end{remark}
%-----------------------------
%
Substituting \eqref{eq:Lag-x} and \eqref{eq:uN} in $f(x,u)$,  and evaluating $f$ over $\pi^N$ we get,
$$f(x^N(\tau), u^N(\tau))\big|_{\pi^N} = [f(x_0, u_0), f(x_1, u_1), \ldots, f(x_N, u_N)]  $$
It is convenient to reuse $f$ as an overloaded function (see \cite{ross-book}, page 8) defined by,
$$ f(X,U) := [f(x_0, u_0), f(x_1, u_1), \ldots, f(x_N, u_N)]^T$$
where,
\begin{equation}
U := [u_0, u_1, \ldots, u_N]^T
\end{equation}
Setting $\dot X = f(X,U)$ and using \eqref{eq:XDOT=DX}, we define Problem $P^N$ according to,
\begin{eqnarray}
&X \in \real{N+1}, \quad U \in \real{N+1}   & \nonumber\\
& \left(\textsf{$P^N$}\right) \left\{
\begin{array}{lrl}
\emph{Minimize } & J^N[X, U] :=& E(x_0,x_N)\\
\emph{Subject to}& \bD X =& f(X, U)  \\
& e(x_0, x_N)  = & 0
\end{array} \right. & \label{eq:ProbPN}
\end{eqnarray}
Comparing \eqref{eq:probP} and \eqref{eq:ProbPN}, it is clear that Problem $P^N$ is strikingly similar to Problem $P$. In fact,  one may view the production of Problem $P^N$ as simple as the process of replacing the continuous operator $d/dt$ in Problem $P$ by the matrix $\bD$.
\begin{remark}
In the derivation of Problem $P^N$, no assumption was made on the structure of the grid~$\pi^N$; hence, \eqref{eq:ProbPN} represents a discretization of Problem $P$ for all PS methods\cite{arb-grid}.
\end{remark}
\begin{remark}\label{remark:cluster}
Once $\pi^N$ is chosen for a given Problem $P$, the theoretical performance of a PS method depends only on the selection of the grid.  The worst performance (i.e., rapid divergence) is characterized by a uniform distribution of nodes while the best performance (i.e., rapid convergence) requires node clustering at the end points\cite{arb-grid}.
\end{remark}
Remark \ref{remark:cluster} is illustrated in Fig.~\ref{fig:arbGridErrors}.
%
%======================================================================================
   \begin{figure}[h!]
      \centering
      \begin{subfigure}[ ]
     {\includegraphics[angle=0, width = 0.8\textwidth]{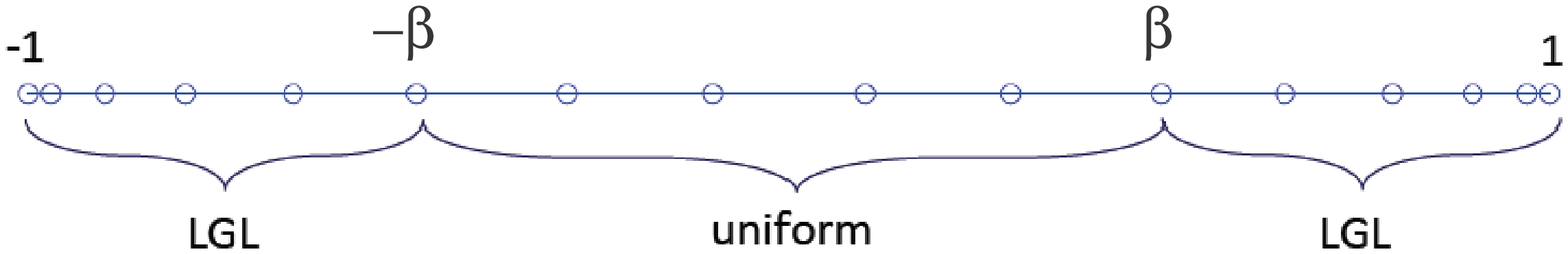}}
     \end{subfigure}
     \begin{subfigure}[]
     {\includegraphics[angle=0, width = 0.6\textwidth]{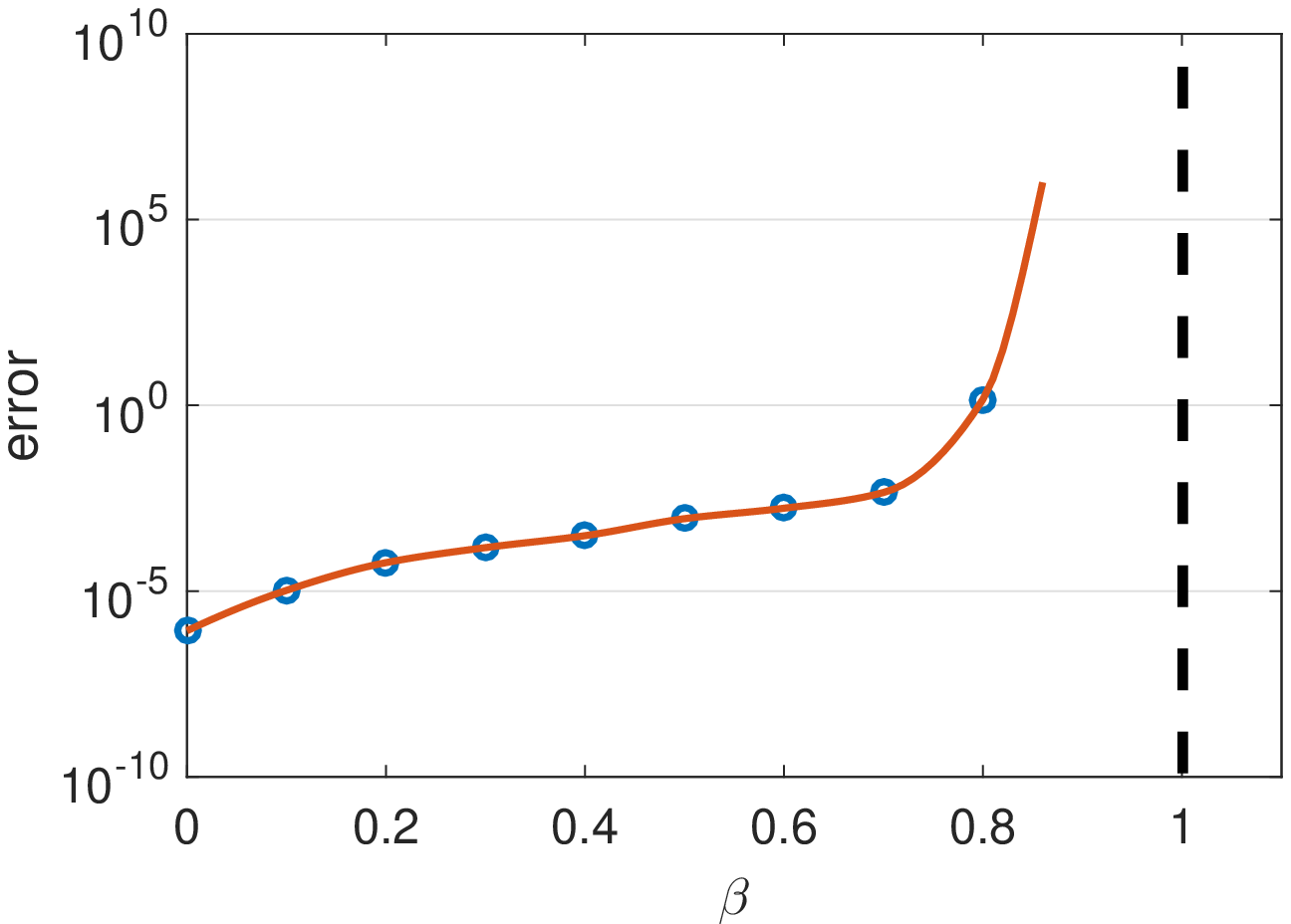}}
     \end{subfigure}
    \caption{\textsf{Effect of node clustering on convergence of PS optimal control\cite{arb-grid}.}}
    \label{fig:arbGridErrors}
   \end{figure}
%==========================================================================================
%
This figure was generated using the data and results from \cite{arb-grid}. The parameter $\beta$ (see Fig.~\ref{fig:arbGridErrors}~a)) denotes a measure of clustering: from
optimal clustering (i.e., when $\beta = 0$) to a uniform node distribution (i.e., when $\beta = 1$).  It is apparent from Fig.~\ref{fig:arbGridErrors} b) that
optimal clustering is not required for convergence but is indeed critical for the rate of convergence.  Hence, it is desirable to maintain node clustering for a fast mesh refinement.

\begin{remark}
If $\pi^N$ is a convergent grid, many of the properties of the continuous-time operator $d/dt$ are faithfully represented by $\bD$.  For instance, let $z$ be a function that is not identically equal to zero; then,
\begin{subequations}
\begin{align}
\text{if } \frac{d}{dt} (z) &= 0  \Rightarrow  \bD Z  = \bzero \label{eq:D-singular}\\
\text{if } \frac{d^2}{dt^2} (z) &= 0  \Rightarrow  \bD^2 Z  = \bzero \label{eq:D-square}
\end{align}
\end{subequations}
\end{remark}
\emph{\textbf{From \eqref{eq:D-singular}, it follows that $\bD$ must be singular to accurately represent $d/dt$.  Note also that $\bD$ must be square so that higher-order derivatives can be properly represented analogous to a repetitive application of $d/dt$ as implied in \eqref{eq:D-square}.}}

\begin{remark}
Node clustering is necessary for polynomial basis functions.  It is possible to reduce node clustering through the use of non-polynomial basis functions\cite{boyd,kt,kt-map}; however, the condition number of the resulting linear algebra still grows with an increase in $N$.
\end{remark}
Problem $P^N$ is solved by the spectral algorithm\cite{PSReview-ARC-2012,spec-alg}.  A key procedure for mesh refinement in the spectral algorithm for PS methods can be described as follows:
\begin{enumerate}
\item Select an initial sequence of integers $ N_0 < N_1 < \cdots < N_m$ and generate a sequence of problems $P^{N_k}, k = 0, 1, \ldots, m $.
\item Solve Problem $N_0$ using the relaxed elastic programming concept described in \cite{spec-alg}.  Set $i=1$.
\item Solve Problem $N_i$ using the results from Problem $N_{i-1}$ by a joint application of PS convergence theory for optimal control\cite{kang-ijrc,Kang_2008_convergence,kang-rate} and the covector mapping principle\cite{ross-book}.
\item Evaluate the (decay of the) spectral coefficients for convergence and stopping criteria\cite{auto-knots,spec-alg}. Exit if all stopping criteria are met; else, set $i = i+1$, and go to Step 3.
\end{enumerate}
\emph{\textbf{Many of the advancements in the spectral algorithm are implemented in DIDO$^\copyright$\cite{scaling,ross-book}, a state-of-the-art MATLAB$^\circledR$ optimal control toolbox for solving optimal control problems.}}
As noted earlier, the condition number of the linear matrix equations involved in the spectral algorithm grows as $N_k^2, k = 0, 1, \ldots, m$; hence, there is a loss in spectral accuracy and convergence rate on those problems that require a finer mesh.  Among many, one of the current remedies for this problem is to use a pre-conditioner, $\bM$, such that the discretized differential equation in Problem $P^N$ can be written as,
\begin{equation}\label{eq:MDX}
\bM \bD X = \bM f(X,U)
\end{equation}
The matrix $\bM$ must be chosen so that the condition number of the matrix equations resulting from the product $\bM \bD$ is lower than when $\bM$ is the identity matrix.  If $\bM$ is not a diagonal matrix, then the sparsity of the right-hand-side of \eqref{eq:MDX} decreases leading to an increase in the computational burden of the resulting linear algebra. Yet another remedy for managing the growth in the condition number is to use PS knots\cite{knots}.  In fact, the full spectral algorithm incorporates an automatic determination of the number and location of knots based on a detection of rapid changes in the control function \cite{auto-knots,spec-alg}.  As noted in Section I, this approach is far from satisfactory.  Although pre-conditioning and PS knots can be used jointly to mitigate the side-effects of both concepts, the resulting algorithm involves too many tuning parameters and inconsistent performance results.  Clearly, there has been a lot of room for improvement in PS optimal control methods, particularly for those problems that require fast and accurate solutions over a fine mesh.  As will be apparent shortly, the use of Birkhoff interpolating polynomials\cite{lorentz,schoenberg} offers a new PS approach where solutions over thousands of points can be generated in a fast and stable manner.

\section{Birkhoff Interpolation Over an Arbitrary Grid}
In contrast to a Lagrange interpolant which is based on interpolating the values of a function, a Birkhoff interpolant is based on interpolating the  values of a heterogenous mix of various orders of derivatives of a function\cite{lorentz}.  That is, a Birkhoff interpolant generalizes Lagrange and Hermite interpolants\cite{schoenberg}.  For second-order boundary value problems (BVPs) with Dirichlet boundary conditions, a Birkhoff interpolation polynomial $\widetilde{x}^N(\tau)$  is given by\cite{wang},
\begin{equation}\label{eq:Wang-BVP-interp}
\widetilde{x}^N(\tau) := x(\tau_0)B_0(\tau) + \sum_{j=1}^{N-1} \ddot x(\tau_j) B_j(\tau) + x(\tau_N)B_N(\tau)
\end{equation}
where,  $[\tau_0, \tau_1, \ldots, \tau_N] :=\pi^N $ is an arbitrary grid as before, and $B_k, k = 0, \ldots, N$ are polynomials of order $N$ or less such that, if they exist, must satisfy the interpolation requirements,
\begin{align}\label{eq:BVP-2order-conditions}
\widetilde{x}^N(\tau_0) = x(\tau_0), \quad \widetilde{x}^N(\tau_N) = x(\tau_N), \quad \ddot{\widetilde{x}}^N(\tau_j) = \ddot x(\tau_j), \quad j = 1, \ldots, N-1
\end{align}
Equation \eqref{eq:BVP-2order-conditions} generates the following conditions for the Birkhoff polynomials,
\begin{equation}\label{eq:BVP-Birk-conditions}
\begin{aligned}
B_0(\tau_0) &= 1        &B_N(\tau_0) &= 0       &B_j(\tau_0) &= 0, &j = 1, \ldots, N-1\\
B_0(\tau_N) & = 0        &B_N(\tau_N) & = 1       &B_j(\tau_N) & = 0, &j = 1, \ldots, N-1 \\
\ddot B_0(\tau_i) & = 0 &\ddot B_N(\tau_i) &= 0  &\ddot B_j(\tau_i) &= \delta_{ij} & i = 1, \ldots, N-1; j =1, \ldots, N-1
\end{aligned}
\end{equation}
That is, for a second-order BVP, a Birkhoff interpolating polynomial may be based on a mix of the second derivatives of a function and its values at the boundary points.
Because an optimal control problem also generates a BVP (through an application of the necessary conditions), it seems reasonable to simply rewrite \eqref{eq:Wang-BVP-interp} for first-order systems so that state-space representation may be used. This simple adaptation generates a proposal for a Birkhoff interpolant given by,
\begin{equation}\label{eq:Wang-BVP-interp-adapt}
\overline{x}^N(\tau) := x(\tau_0)B_0(\tau) + \sum_{i=1}^{N-1} \dot x(\tau_i) B_i(\tau) + x(\tau_N)B_N(\tau)
\end{equation}
where, $B_k, k = 0, \ldots, N$ are polynomials (possibly different from those in \eqref{eq:Wang-BVP-interp}) that must satisfy,
\begin{align}\label{eq:BVP-1order-conditions}
\overline{x}^N(\tau_0) = x(\tau_0), \quad \overline{x}^N(\tau_N) = x(\tau_N), \quad \dot{\overline{x}}^N(\tau_i) = \dot x(\tau_i)
\end{align}
The conditions for \eqref{eq:BVP-1order-conditions} that are analogous to \eqref{eq:BVP-Birk-conditions} are given by,
\begin{equation}\label{eq:BVP-Birk-conditions-adapt}
\begin{aligned}
B_0(\tau_0) &= 1        &B_N(\tau_0) &= 0       &B_j(\tau_0) &= 0, &j = 1, \ldots, N-1\\
B_0(\tau_N) & = 0        &B_N(\tau_N) & = 1       &B_j(\tau_N) & = 0, &j = 1, \ldots, N-1 \\
\dot B_0(\tau_i) & = 0 &\dot B_N(\tau_i) &= 0  &\dot B_j(\tau_i) &= \delta_{ij} & i = 1, \ldots, N-1; j =1, \ldots, N-1
\end{aligned}
\end{equation}
The satisfaction of the conditions given by \eqref{eq:BVP-Birk-conditions-adapt} has some well-known problems related to the existence and uniqueness of a solution\cite{lorentz,schoenberg}; hence, we do not pursue this further.  If the conditions on $B_0$ or $B_N$ are removed from \eqref{eq:BVP-Birk-conditions-adapt}, then it can be shown that a unique Birkhoff interpolant exists\cite{finden}.
In doing the latter, a Birkhoff interpolant is given by,
\begin{align}\label{eq:Wang-IVP-interp}
\textbf{Case (a):}  && x^N(\tau) := x(\tau_0)B_0(\tau) + \sum_{j=1}^N \dot x(\tau_j) B_j(\tau)
\end{align}
where, $B_k, k = 0, \ldots, N$ must now only satisfy,
\begin{align}\label{eq:IVP-1order-conditions}
x^N(\tau_0) = x(\tau_0),  \quad \dot x^N(\tau_j) = \dot x(\tau_j), \quad j = 1, \ldots, N
\end{align}
Equation \eqref{eq:IVP-1order-conditions} generates fewer conditions than \eqref{eq:BVP-Birk-conditions-adapt}; these are given by,
\begin{equation}\label{eq:birk-case-a- conditions}
\begin{aligned}
&B_0(\tau_0) &= 1,            &&& B_j(\tau_0) &&= 0,                      \qquad & j = 1, \ldots, N \\
&\dot B_0(\tau_i) & = 0,  &&& \dot B_j(\tau_i) && = \delta_{ij},  & i = 1, \ldots, N
\end{aligned}
\end{equation}
In the same manner, an equivalent Birkhoff interpolant is given by,
\begin{align}\label{eq:Birk-interp-FVP}
\textbf{Case (b):} && x^N(\tau) := \sum_{j=0}^{N-1} \dot x(\tau_j) B_j(\tau) + x(\tau_N)B_N(\tau)
\end{align}
where, $B_k, k = 0, \ldots, N$ are yet another set of polynomials that must satisfy,
\begin{align}\label{eq:FVP-1order-conditions}
x^N(\tau_N) = x(\tau_N),  \quad \dot x^N(\tau_j) = \dot x(\tau_j), \quad j = 0, \ldots, N-1
\end{align}
The conditions for  $B_k, k = 0, \ldots, N$ in \eqref{eq:Birk-interp-FVP} are given by,
\begin{equation}\label{eq:birk-case-b-conditions}
\begin{aligned}
&B_N(\tau_N) &= 1,             &&& B_j(\tau_N) &&= 0,                       \qquad & j = 0, \ldots, N-1 \\
&\dot B_N(\tau_i) & = 0,  &&& \dot B_j(\tau_i) && = \delta_{ij},  & i = 0, \ldots, N-1
\end{aligned}
\end{equation}
It is easy to show that a Birkhoff interpolant exists for both cases given by \eqref{eq:Wang-IVP-interp} and \eqref{eq:Birk-interp-FVP}.
\begin{remark}
The Birkhoff interpolants given by \eqref{eq:Wang-IVP-interp} and \eqref{eq:Birk-interp-FVP} are over an arbitrary grid, $\pi^N$.  Endpoint node clustering is necessary for convergence (see Fig.~\ref{fig:arbGridErrors}). Consequently, Birkhoff interpolants for Jacobi-Gauss-Lobatto, Jacobi-Gauss-Radau, and Jacobi-Gauss meshes follow from \eqref{eq:Wang-IVP-interp} and \eqref{eq:Birk-interp-FVP} by an appropriate choice of $\pi^N$.

\end{remark}

It will be apparent shortly that it is convenient to define two subsets of the grid $\pi^N$ given by,
\begin{subequations}
\begin{align}
\pi^N_a &:= [\tau_1, \ldots, \tau_N]\\
\pi^N_b &:= [\tau_0, \tau_1, \ldots, \tau_{N-1}]
\end{align}
\end{subequations}
Thus, $\pi^N$ may be represented either by $\pi^N = [\tau_0, \pi^N_a]$ or $\pi^N = [\pi^N_b, \tau_N]$. Furthermore, we define the matrices,
\begin{equation}\label{eq:DBdef4ab}
\begin{aligned}
\bD_a &:= \Big(\dot L_j(\tau_i)  \Big)_{1\le i,j \le N}  &\bB_a &:= \Big( B_j(\tau_i)  \Big)_{1\le i,j \le N}\\
\bD_b &:= \Big(\dot L_j(\tau_i)  \Big)_{0\le i,j \le N-1}  &\bB_b &:= \Big( B_j(\tau_i)  \Big)_{0\le i,j \le N-1}
\end{aligned}
\end{equation}
where, $\bB_a$ and $\bB_b$ satisfy \eqref{eq:birk-case-a- conditions} and \eqref{eq:birk-case-b-conditions} respectively.
%
%========================================================================
\begin{theorem}[Wang-Samson-Zhao]\label{theorem:WSZ}
Let $\omega \in \set{a, b}$. Then,
$$\bD_\omega \bB_\omega = \bI_N  $$
where, $\bI_N$ is an $N\times N$ identity matrix.
\end{theorem}
%======================================================================
%
\begin{proof}
The proof of this result for $\omega = a$ is implied in \cite{wang} by Theorems 3.2 and 4.1.  For the purposes of completeness and clarity, we prove this theorem for an arbitrary grid by the same procedures used in \cite{wang}.  Furthermore, because the matrix $\bB_b$ is not presented in \cite{wang}, we first prove this result for $\omega = b$.  The proof for $\omega = a$ follows by a trivial extension.

For any $N^{th}$-order polynomial $\phi$, we can write,
\begin{equation}\label{eq:proof-1}
\phi(\tau) = \sum_{k=0}^N \phi(\tau_k) L_k(\tau) \Rightarrow \dot\phi(\tau) = \sum_{k=0}^N \phi(\tau_k) \dot L_k(\tau)
\end{equation}
Setting $\phi(\tau) = B_j(\tau)$ in \eqref{eq:proof-1} and evaluating the resulting equation at $\tau_i$, we get,
\begin{equation}\label{eq:proof-2}
\dot B_j(\tau_i) = \sum_{k=0}^N B_j(\tau_k) \dot L_k(\tau_i)
\end{equation}
We can rewrite \eqref{eq:proof-2} for Case (b) evaluation as,
\begin{equation}\label{eq:proof-3}
\dot B_j(\tau_i) = \sum_{k=0}^{N-1} B_j(\tau_k) \dot L_k(\tau_i) + B_j(\tau_N) \dot L_N(\tau_i)
\end{equation}
Using \eqref{eq:birk-case-b-conditions} to evaluate \eqref{eq:proof-3} we get,
\begin{equation}\label{eq:proof-4}
\delta_{ij} = \sum_{k=0}^{N-1} B_j(\tau_k) \dot L_k(\tau_i), \quad 0 \le i, j \le N-1
\end{equation}
From \eqref{eq:DBdef4ab}, it is clear that \eqref{eq:proof-4} is the same as $\bD_\omega \bB_\omega = \bI_N  $ for $\omega = b$.  It is also apparent that the proof for $\omega = a$ follows by a similar procedure.
\end{proof}

\section{A Family of Well-Conditioned PS Optimal Control Methods}

In computational optimal control, it is apparent that it is important to generate a reasonably accurate control trajectory $[t_0, t_f] \ni t \mapsto \bu$.  Furthermore, it is more important to generate an accurate value of the initial control than its final value because any error in $\bu(t_0)$ will be amplified in terms of propagating the state trajectory through the dynamics $\dot\bx = \bff(\bx, \bu(t))$; see \cite{scaling,ross-book} for details.  This implies that the constraint at the initial point $\dot\bx(t_0) = \bff(\bx(t_0), \bu(t_0))$ must be well-represented. In \cite{wang}, initial value problems (IVPs) are solved using the matrix $\bB_a$ and the differential equation is not imposed at the initial point. Because this approach is not satisfactory for optimal control applications, we need to modify existing results on Birkhoff interpolation for optimal control applications.

\subsection{Birkhoff PS Optimal Control Method Over an Arbitrary Grid Based on \eqref{eq:Wang-IVP-interp} }
In applying the results of the previous section to Problem $P$, we set,
\begin{equation}\label{eq:birk-a}
x^N(\tau) :=x_0B_0(\tau) + \sum_{j=1}^N v_j B_j(\tau)
\end{equation}
where, $x_0$ and $v_j, j = 1, \ldots, N$ are the unknown optimization variables over an arbitrary grid $\pi^N$, and $B_k, k = 0, 1, \ldots, N$ satisfy \eqref{eq:birk-case-a- conditions}.
Differentiating both sides of \eqref{eq:birk-a} we get,
\begin{equation}\label{eq:dotxN-birk-a}
\dot x^N(\tau) = x_0\dot B_0(\tau) + \sum_{j=1}^N v_j \dot B_j(\tau)
\end{equation}
Substituting \eqref{eq:dotxN-birk-a} in the dynamic constraint $\dot x^N(\tau) = f(x^N(\tau), u^N(\tau))$, we get,
\begin{equation}\label{eq:birk-a-xdot-forall}
x_0\dot B_0(\tau) + \sum_{j=1}^N v_j \dot B_j(\tau) = f(x^N(\tau), u^N(\tau))
\end{equation}
Evaluating \eqref{eq:birk-a-xdot-forall} over the grid $\pi^N = [\tau_0, \pi^N_a]$ we get,
\begin{align}
I_a \, V_a &= f(x_0, u_0) - x_0 \dot B_0(\tau_0) \label{eq:IaVa=}\\
V_a &= f(X_a, U_a)\label{eq:Va=f}
\end{align}
where, $I_a, V_a$ and $X_a$ are given by,
\begin{equation}\label{eq:a-var-def}
\begin{aligned}
I_a &:= [\dot B_1(\tau_0), \ldots, \dot B_N(\tau_0)] \\
V_a &:= [v_1, \ldots, v_N]^T\\
X_a &:= [x^N(\tau_1), \ldots, x^N(\tau_N)]^T
\end{aligned}
\end{equation}
The vector $X_a$ can be evaluated using \eqref{eq:birk-a}; this generates the constraint,
\begin{equation}\label{eq:Xa-eval}
X_a = x_0\, \bb_0 + \bB_a V_a
\end{equation}
where,
\begin{equation}\label{eq:b0-def}
\bb_0 := \left(
                  \begin{array}{c}
                    B_0(\tau_1) \\
                    B_0(\tau_2) \\
                    \vdots \\
                    B_0(\tau_N) \\
                  \end{array}
                \right)
\end{equation}
Collecting all the relevant equations, we define
\begin{eqnarray}
&X \in \real{N+1}, \quad U \in \real{N+1}, \quad V_a \in \real{N}   & \nonumber\\
& (\textsf{$P_a^N$}) \left\{
\begin{array}{lrl}
\emph{Minimize } & J^N_a[X, U, V_a] :=& E(x_0,x_N)\\
\emph{Subject to}& V_a =& f(X_a, U_a)  \\
& X_a  = &x_0\, \bb_0 + \bB_a V_a \\
& I_a \, V_a =& f(x_0, u_0) - x_0 \dot B_0(\tau_0)\\
& e(x_0, x_N)  = & 0
\end{array} \right. & \label{eq:ProbPNa}
\end{eqnarray}
\begin{remark}
The Birkhoff equality constraint $I_a \, V_a + x_0 \dot B_0(\tau_0)- f(x_0, u_0) = 0$ in Problem $P^N_a$ imposes the differential equation $\dot x(\tau) = f(x(\tau), u(\tau)) $  at the initial point $\tau_0$.  This Birkhoff constraint is equivalent to imposing the same differential equation via the Lagrange condition $\sum_{j=0}^N x_j \dot L_j (\tau_0) - f(x^N(\tau_0), u^N(\tau_0)) = 0$.  This equivalency is proved as Proposition \ref{proposition-xdot0}.
\end{remark}
%
%===============================
\begin{proposition}\label{proposition-xdot0}
Let $x_j = x^N(\tau_j)$. At $\tau = \tau_0$, the Lagrange and Birkhoff interpolants satisfy the condition,
\begin{equation}
\sum_{j=0}^N x_j \dot L_j (\tau_0) = x_0 \dot B_0(\tau_0) + \sum_{j=1}^N v_j \dot B_j(\tau_0) \label{eq:xdot0-a}
\end{equation}
\end{proposition}
%-----------------------
%
\begin{proof}
The derivative $\dot B_j(\tau_i)$ can be obtained by an application of the PS differentiation matrix; hence, we have,
\begin{equation}\label{eq:dotB=DB}
\dot B_j(\tau_i) = \sum_{k=0}^N \dot L_k(\tau_i) B_j(\tau_k)
\end{equation}
Setting $i = j = 0$ in \eqref{eq:dotB=DB} we get,
\begin{align}
\dot B_0(\tau_0) &= \sum_{k=0}^N \dot L_k(\tau_0) B_0(\tau_k) \nonumber \\
                 & = \dot L_0(\tau_0) B_0(\tau_0) + \sum_{k=1}^N \dot L_k(\tau_0) B_0(\tau_k) \nonumber\\
                 & = \dot L_0(\tau_0) + \sum_{k=1}^N \dot L_k(\tau_0) B_0(\tau_k)\label{eq:dotB00}
\end{align}
where, the last equality in \eqref{eq:dotB00} follows from the fact that $B_0(\tau_0) = 1$; see \eqref{eq:birk-case-a- conditions}.

Similarly, setting $i = 0$ in \eqref{eq:dotB=DB} we get,
\begin{align}
\dot B_j(\tau_0) &= \sum_{k=0}^N \dot L_k(\tau_0) B_j(\tau_k) \nonumber\\
                &= \dot L_0(\tau_0) B_j(\tau_0)  + \sum_{k=1}^N \dot L_k(\tau_0) B_j(\tau_k) \nonumber \\
                & = \sum_{k=1}^N \dot L_k(\tau_0) B_j(\tau_k) \label{eq:dotB0j}
\end{align}
Evaluating the right-hand-side of \eqref{eq:xdot0-a} using \eqref{eq:dotB00} and \eqref{eq:dotB0j} we have,
\begin{equation}\label{eq:rhs-a}
x_0 \dot B_0(\tau_0) + \sum_{j=1}^N v_j \dot B_j(\tau_0) =  x_0\dot L_0(\tau_0) + \sum_{k=1}^N x_0 \dot L_k(\tau_0) B_0(\tau_k) + \sum_{j=1}^N v_j \sum_{k=1}^N \dot L_k(\tau_0) B_j(\tau_k)
\end{equation}
Next, from \eqref{eq:birk-a}, we have,
\begin{equation}\label{eq:birk-a-redux}
x_k = x_0 B_0(\tau_k) + \sum_{j=1}^N v_j B_j(\tau_k)
\end{equation}
Evaluating the left-hand-side of \eqref{eq:xdot0-a} using \eqref{eq:birk-a-redux} we get,
\begin{align}
\sum_{k=0}^N x_k \dot L_k (\tau_0) &= x_0 \dot L_0 (\tau_0) + \sum_{k=1}^N x_k \dot L_k (\tau_0) \nonumber\\
&= x_0 \dot L_0 (\tau_0) +  \sum_{k=1}^N x_0 B_0(\tau_k) \dot L_k(\tau_0) + \sum_{k=1}^N \sum_{j=1}^N v_j B_j(\tau_k) \dot L_k(\tau_0) \label{eq:lhs-a}
\end{align}
It is evident that the right-hand-sides of \eqref{eq:rhs-a} and \eqref{eq:lhs-a} are identical.
\end{proof}
%=============================
%

\subsection{Birkhoff PS Optimal Control Method Over an Arbitrary Grid Based on \eqref{eq:Birk-interp-FVP}}

Although many of the results are similar to those resulting from \eqref{eq:birk-a}, we provide a brief development here to highlight the key differences between the two approaches.  To this end, we set,
\begin{equation}\label{eq:birk-b}
x^N(\tau) :=\sum_{j=0}^{N-1} v_j B_j(\tau) + x_N B_N(\tau)
\end{equation}
Substituting the \eqref{eq:birk-b} in the dynamic constraint we get,
\begin{equation}\label{eq:dotxN-birk-b}
\sum_{j=0}^{N-1} v_j \dot B_j(\tau) + x_N\dot B_N(\tau) = f(x^N(\tau), u^N(\tau))
\end{equation}
Evaluating \eqref{eq:dotxN-birk-b} over the grid $\pi^N = [\pi^N_b, \tau_N]$ we get,
\begin{align}
V_b &= f(X_b, U_b)\\
I_b \, V_b &= f(x_N, u_N) - x_N \dot B_N(\tau_N)
\end{align}
where, $I_b, V_b$ and $X_b$ are given by,
\begin{equation}
\begin{aligned}
I_b &:= [\dot B_0(\tau_N), \ldots, \dot B_{N-1}(\tau_N)] \\
V_b &:= [v_0, \ldots, v_{N-1}]^T\\
X_b &:= [x^N(\tau_0), \ldots, x^N(\tau_{N-1})]^T
\end{aligned}
\end{equation}
Using \eqref{eq:birk-b} to evaluate $X_b$ we get,
\begin{equation}\label{eq:Xb-eval}
X_b = \bB_b V_b + x_N\, \bb_N
\end{equation}
where,
\begin{equation}\label{eq:bN-def}
\bb_N := \left(
                  \begin{array}{c}
                    B_N(\tau_0) \\
                    B_N(\tau_1) \\
                    \vdots \\
                    B_N(\tau_{N-1}) \\
                  \end{array}
                \right)
\end{equation}
Thus, we now have a new discretization of Problem $P$ given by,
\begin{eqnarray}
&X \in \real{N+1}, \quad U \in \real{N+1}, \quad V_b \in \real{N}   & \nonumber\\
& (\textsf{$P_b^N$}) \left\{
\begin{array}{lrl}
\emph{Minimize } & J^N_b[X, U, V_b] :=& E(x_0,x_N)\\
\emph{Subject to}& V_b =& f(X_b, U_b)  \\
& X_b  = &\bB_b V_b  + x_N\, \bb_N  \\
& I_b \, V_b =& f(x_N, u_N) - x_N \dot B_N(\tau_N)\\
& e(x_0, x_N)  = & 0
\end{array} \right. & \label{eq:ProbPNb}
\end{eqnarray}
\begin{remark}
In Problem $P^N_b$, the differential constraint $\dot x(\tau) - f(x(\tau), u(\tau)) = 0$  at the initial point $\tau_0$ is imposed via the constraint $V_b-f(X_b, U_b) = 0$.
The Birkhoff equality constraint $I_b \, V_b + x_N \dot B_N(\tau_N)- f(x_N, u_N) = 0$  imposes the differential constraint at the final time.  This Birkhoff condition is equivalent to imposing the same differential constraint via the Lagrange condition $\sum_{j=0}^N x_j \dot L_j (\tau_N) - f(x^N(\tau_N), u^N(\tau_N)) =0$.  The proof of this equivalency is similar to Proposition 1; hence, it is omitted.
\end{remark}

\subsection{Several Pre-Conditioned Lagrange PS Optimal Control Methods}
It is possible to develop several pre-conditioned Lagrange PS optimal control methods based on using  Birkhoff matrices as pre-conditioners.  We first consider the case where we partition the mesh $\pi^N$ as $[\tau_0, \pi^N_a]$.

\subsubsection{Left-Preconditioned Lagrange PS Method Based on $[\tau_0, \pi^N_a]$ Partitioning }
It is convenient to rewrite \eqref{eq:Lag-x} as,
\begin{equation}\label{eq:Lag-x-a}
x^N(\tau):=  x_0 L_0(\tau) + \sum_{i=1}^{N} x_i L_i(\tau)
\end{equation}
Differentiating \eqref{eq:Lag-x-a} and substituting the result in the dynamic constraint $\dot x = f(x, u)$, we get,
\begin{equation}\label{eq:Lag-col-a}
x_0\dot L_0(\tau) + \sum_{i=1}^{N} x_i \dot L_i(\tau) = f(x^N(\tau), u^N(\tau))
\end{equation}
Evaluating \eqref{eq:Lag-col-a} over the grid $\pi^N = [\tau_0, \pi^N_a]$ we get,
\begin{subequations}\label{eq:DX=f-case-a}
\begin{align}
\sum_{i=0}^{N} x_i \dot L_i(\tau_0) &= f(x_0, u_0) \label{eq:DX=f-case-a-zero}\\
\bD_a X_a &= f(X_a, U_a) - x_0\,\bl_0   \label{eq:DX=f-case-a-int}
\end{align}
\end{subequations}
where, $X_a$ and $U_a$ are defined in \eqref{eq:a-var-def} and $\bl_0$ is defined by,
\begin{equation}\label{eq:l0-def}
\bl_0 := \left(
          \begin{array}{c}
            \dot L_0(\tau_1) \\
            \dot L_0(\tau_2) \\
            \vdots \\
            \dot L_0(\tau_N) \\
          \end{array}
        \right)
\end{equation}
Note that \eqref{eq:DX=f-case-a} is merely a partitioned version of the discrete differential constraint stipulated in Problem $P^N$ defined in \eqref{eq:ProbPN}. Pre-multiplying both sides of \eqref{eq:DX=f-case-a-int} by $\bB_a$ and using Theorem \ref{theorem:WSZ} we get,
\begin{equation}\label{eq:DX=f-left-pre}
X_a = \bB_a f(X_a, U_a) - x_0 \bB_a\bl_0
\end{equation}
As a result of \eqref{eq:DX=f-left-pre}, Problem $P^N$ transforms according to,
\begin{eqnarray}
&X \in \real{N+1}, \quad U \in \real{N+1}   & \nonumber\\
& (\text{Left-$P^N_a$}) \left\{
\begin{array}{lrl}
\emph{Minimize } & J^N[X, U] :=& E(x_0,x_N)\\
\emph{Subject to} &\sum_{i=0}^{N} x_i \dot L_i(\tau_0) =& f(x_0, u_0)\\
& X_a = & \bB_a f(X_a, U_a) - x_0 \bB_a\bl_0 \\
& e(x_0, x_N)  = & 0
\end{array} \right. & \label{eq:left-ProbPNa}
\end{eqnarray}

\subsubsection{Right-Preconditioned Lagrange PS Method Based on $[\tau_0, \pi^N_a]$ Partitioning }
Instead of pre-multiplying equations to achieve better conditioning, we can generate an alternative method by post-multiplying only the left-hand-side of \eqref{eq:DX=f-case-a-int} by $B_a$.  To clarify this point, we substitute \eqref{eq:Xa-eval} in \eqref{eq:DX=f-case-a-int}; this results in,
\begin{equation}\label{eq:Va=f+stuff}
V_a + x_0 \bD_a \bb_0 = f(X_a, U_a) - x_0\bl_0
\end{equation}
Comparing \eqref{eq:Va=f+stuff} with \eqref{eq:Va=f}, it is evident that we need to prove that $ x_0 \bD_a \bb_0 = - x_0\bl_0$.  This result is proved as Proposition \ref{proposition:stuff=0}.
\begin{proposition}\label{proposition:stuff=0}
$ \bD_a \bb_0 + \bl_0 = \bzero$
\end{proposition}
\begin{proof}
From \eqref{eq:dotB=DB}, we have,
\begin{equation}\label{eq:dotB=DB-case-a}
\dot B_j(\tau_i) =\dot L_0(\tau_i) B_j(\tau_0) + \sum_{k=1}^N \dot L_k(\tau_i) B_j(\tau_k)
\end{equation}
Evaluating \eqref{eq:dotB=DB-case-a} for $j=0$ and  $i= 1, \ldots N$, we get,
\begin{equation}\label{eq:dotB-0-case-a-int}
\left(
  \begin{array}{c}
    \dot B_0(\tau_1) \\
    \vdots \\
    \dot B_0(\tau_N) \\
  \end{array}
\right) = \left(
  \begin{array}{c}
    \dot L_0(\tau_1) \\
    \vdots \\
    \dot L_0(\tau_N) \\
  \end{array}
\right) B_0(\tau_0) + \bD_a \bb_0
\end{equation}
The left-hand-side of \eqref{eq:dotB-0-case-a-int} is zero from \eqref{eq:birk-case-a- conditions}; hence, the proposition follows from \eqref{eq:l0-def} and the fact that $B_0(\tau_0) = 1$ by definition.
\end{proof}

As a result of Proposition \ref{proposition:stuff=0}, \eqref{eq:Va=f+stuff} simplifies to $V_a = f(X_a, U_a)$. Using Proposition \ref{proposition-xdot0} to replace \eqref{eq:DX=f-case-a-zero}, Problem $P^N$ now transforms to a ``right-preconditioned'' method given by,
\begin{eqnarray}
&X \in \real{N+1}, \quad U \in \real{N+1}, \quad V_a \in \real{N}   & \nonumber\\
& (\text{right-$P_a^N$}) \left\{
\begin{array}{lrl}
\emph{Minimize } & J^N_a[X, U, V_a] :=& E(x_0,x_N)\\
\emph{Subject to} & I_a \, V_a =& f(x_0, u_0) - x_0 \dot B_0(\tau_0)\\
& V_a =& f(X_a, U_a)  \\
& X_a  = &x_0\, \bb_0 + \bB_a V_a \\
& e(x_0, x_N)  = & 0
\end{array} \right. & \label{eq:rightProbPNa}
\end{eqnarray}
It is apparent that \eqref{eq:rightProbPNa} is identical to Problem $P^N_a$ defined by \eqref{eq:ProbPNa}.  In other words, the Birkhoff-PS method (for Case (a)) may be viewed as a right-pre-conditioned Lagrange-PS method.

%------------------
\begin{proposition}
If the variable $V_a$ is eliminated from Problem right-$P^N_a$, it reduces to Problem left-$P^N_a$.
\end{proposition}
%------------------
\begin{proof}
Comparing \eqref{eq:left-ProbPNa} with \eqref{eq:rightProbPNa}, it follows that this proposition is proved if we can show that:
\begin{itemize}
\item[(1)] $\bB_a f(X_a, U_a) - x_0 \bB_a\bl_0 = x_0\, \bb_0 + \bB_a V_a$, and
\item[(2)] the equation $\sum_{i=0}^{N} x_i \dot L_i(\tau_0) = f(x_0, u_0)$ is equivalent to  $I_a \, V_a = f(x_0, u_0) - x_0 \dot B_0(\tau_0)$
\end{itemize}
Hence, we prove this proposition in two parts.

\emph{Part (1):}

Substituting \eqref{eq:Va=f} in \eqref{eq:Xa-eval}, we get,
\begin{align}
x_0\, \bb_0 + \bB_a V_a &= x_0\bb_0 + \bB_a f(X_a, U_a) \nonumber\\
        & =  -x_0\bB_a\bl_0 + \bB_a f(X_a, U_a) \label{eq:part1-qed}
\end{align}
where, the last equality in \eqref{eq:part1-qed} follows from Proposition \eqref{proposition:stuff=0} and Theorem \ref{theorem:WSZ}.

\emph{Part (2):}

The proof of this part is a direct consequence of Proposition \ref{proposition-xdot0}.
\end{proof}
%==================

\begin{remark}
It is clear that two additional left and right pre-conditioned Lagrange-PS methods can be similarly obtained by using the (Case (b)) partitioning $\pi^N = [\pi^N_b, \tau_N]$.  For the purposes of brevity, these cases are not derived.
\end{remark}

\begin{remark}
For Case (a), it is straightforward to show that $B_0(\tau) = 1$ for all $\tau \in [\tau_0, \tau_N]$.  Consequently, $\bb_0$ defined in \eqref{eq:b0-def} is a vector of ones and $\dot B_0(\tau_0) = 0$.  Similarly, for Case (b), $B_N(\tau) = 1$ for all $\tau \in [\tau_0, \tau_N]$; hence, $\bb_N$ defined in \eqref{eq:bN-def} is a vector of ones and $\dot B_N(\tau_N) = 0$.
\end{remark}

\begin{remark}\label{remark:Birkhoff-galore}
All PS methods proposed in this section are for an arbitrary grid.  It is possible to generate a very large sub-class of each of these methods by choosing slightly different and well-known grid points. Popular examples of such grid points are Chebyshev-Gauss-Lobatto (CGL), Legendre-Gauss-Lobatto (LGL), Chebyshev-Gauss-Radau (CGR), Legendre-Gauss-Radau (LGR), Chebyshev-Gauss (CG) and Legendre-Gauss (LG).  See \cite{PSReview-ARC-2012,advances,arb-grid} for further details on advantages and pitfalls associated with the choice of these mesh points as they apply to computational optimal control problems.
\end{remark}
Because it is possible to generate a very large family of PS methods based on the results of the preceding subsections, we limit the scope of the present paper to illustrating only a small set of these methods.  A more extensive discussion of other options and numerical results are discussed in \cite{koeppen}.  Nonetheless, we note that \emph{\textbf{a selection of the proper grid for optimal control problems is largely based on whether or not the horizon is finite or infinite\cite{ross-book}}}.

\section{A Numerical Demonstration of the Well-Conditioning of Birkhoff PS Discretizations}
A quick examination of \eqref{eq:rightProbPNa} shows that the linear equation given by \eqref{eq:Xa-eval} is the only equation in a Birkhoff PS method (for Case (a)) that is independent of the problem data functions $E, e$ and $f$. Consequently, if the original continuous-time problem is well-conditioned, the conditioning of the discretized problem is driven by the condition numbers of the linear equation given by \eqref{eq:Xa-eval} for Case (a), and by \eqref{eq:Xb-eval} for Case (b). Because all conclusions drawn with respect to the condition number of Case (a) are identical to that of Case (b), we focus the entirety  of our discussions in this section to that of Case (a) only.

A very powerful statement on the condition number of the Birkhoff PS discretization can be drawn for the special case of a known value of the initial condition, $x_0$. In this case, we can re-write \eqref{eq:Xa-eval} as,
\begin{equation}\label{eq:Xa-rewrite-1}
\bB_a V_a - X_a = -x_0\, \bb_0
\end{equation}
where, the right-hand-side of \eqref{eq:Xa-rewrite-1} is a known quantity.  Hence, the condition number of a Birkhoff PS discretization is driven by the condition number of the ($N\times 2N$)-matrix $\bC^{\text{Birk}}_N$ defined by,
\begin{equation}\label{eq:CN-Birk-def}
\bC^{\text{Birk}}_N := \big[\bB_a, -\bI_N\big]
\end{equation}
where $\bI_N$ is the $N\times N$ identity matrix. As evident from Fig.~\ref{fig:cond-CN-Birk},
%
%======================================================================================
   \begin{figure}[h!]
      \centering
     \includegraphics[angle=0, width = 0.6\textwidth]{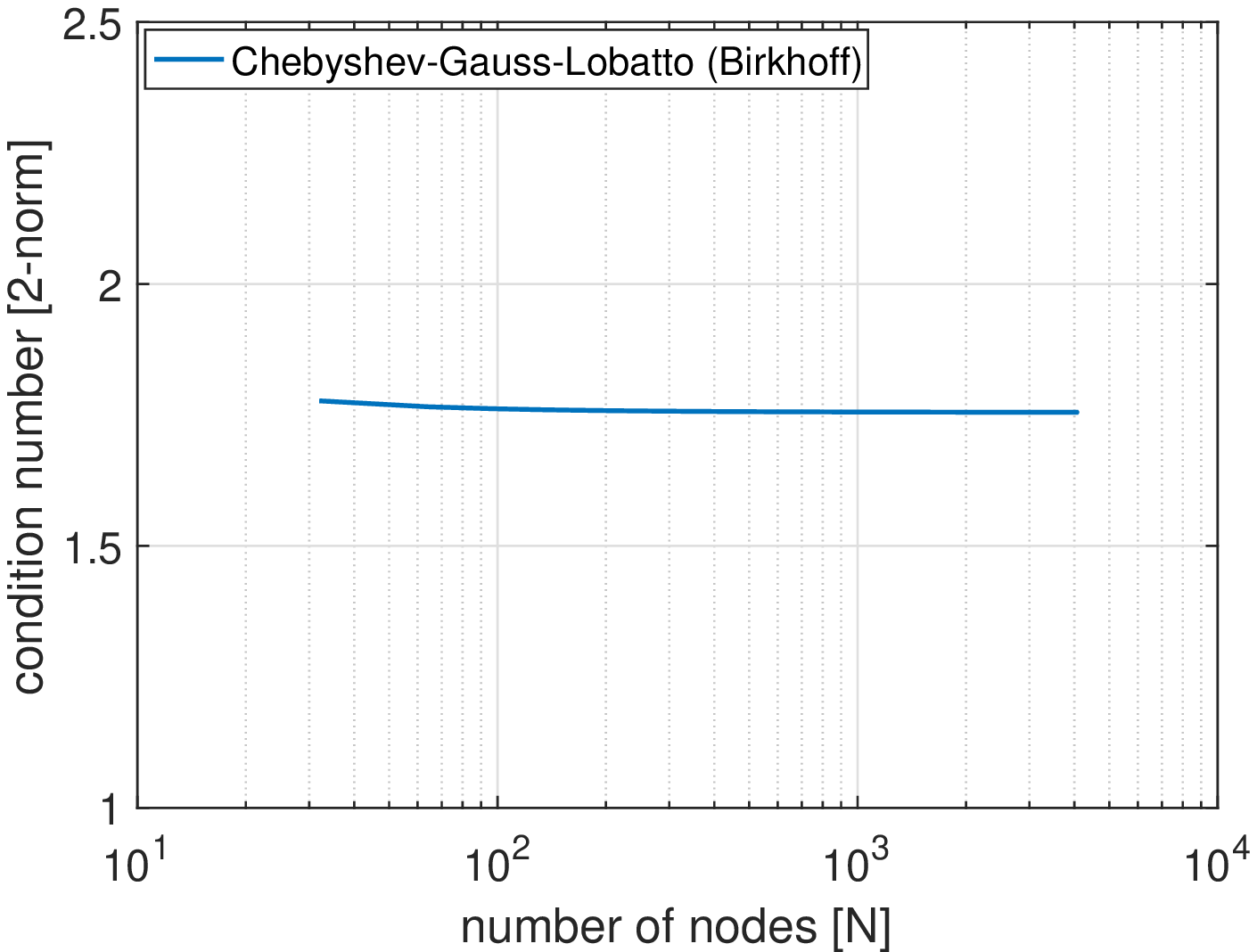}
    \caption{\textsf{Demonstrating $\mathcal{O}(1)$ variation in the condition number of the matrix $\bC^{\text{Birk}}_N$ defined in \eqref{eq:CN-Birk-def} with evaluations over $N$ CGL nodes. }}
    \label{fig:cond-CN-Birk}
   \end{figure}
%==========================================================================================
%
the condition number of $\bC^{\text{Birk}}_N$ is remarkably flat with respect to $N$; i.e., of $\mathcal{O}(1)$.  The condition number of a corresponding Lagrange PS discretization may be obtained by multiplying both sides of \eqref{eq:Xa-rewrite-1} by $\bD_a$ and invoking Theorem 1.  This results in
\begin{equation}\label{eq:Xa-rewrite-Lagrange}
\bD_a X_a - V_a = x_0\, \bD_a \bb_0
\end{equation}
with the right-hand-side of \eqref{eq:Xa-rewrite-Lagrange} considered to be a known quantity. Hence, the condition number of a Lagrange PS discretization is driven by the condition number of the ($N\times 2N$)-matrix,
\begin{equation}\label{eq:CN-Lagr-def}
\bC^{\text{Lagr}}_N := \big[\bD_a, -\bI_N\big]
\end{equation}
It is clear from Fig.~\ref{fig:cond_CN}
%
%======================================================================================
   \begin{figure}[h!]
      \centering
      {\includegraphics[angle=0, width = 0.6\textwidth]{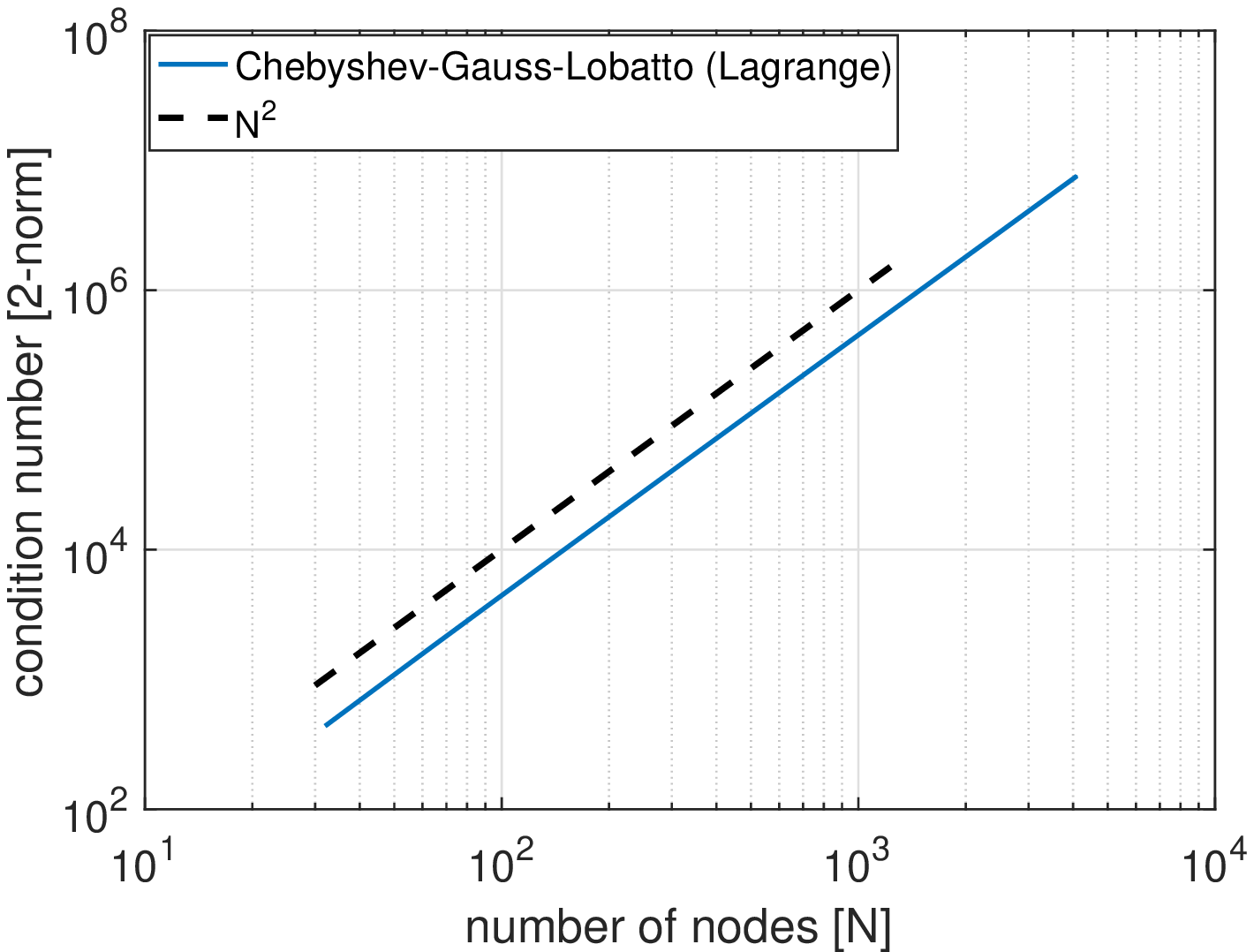}}
      \caption{\textsf{$\mathcal{O}(N^2)$-growth in the condition number of the matrix $C_N^{\text{Lagr}}$ defined in \eqref{eq:CN-Lagr-def} with evaluations over $N$ CGL nodes.}}
    \label{fig:cond_CN}
   \end{figure}
%==========================================================================================
%
that the condition number of $\bC^{\text{Lagr}}_N$ is $\mathcal{O}(N^2)$ and consistent with the plot shown in Fig.~\ref{fig:CondNo}.

In a generic optimal control problem, $x_0$ is an optimization variable.  In this general case, the right-hand-side of \eqref{eq:Xa-rewrite-1} is not known; hence, we need to examine the condition number of the $N\times (2N+1)$-matrix,
\begin{equation}\label{eq:AN-Birk-def}
\bA^{\text{Birk}}_N := \big[\bB_a, -\bI_N, \bb_0\big]
\end{equation}
It is apparent from Fig.~\ref{fig:cond_AN} that the condition number of $\bA_N^{\text{Birk}}$ varies only as $\mathcal{O}\left(\sqrt{N}\right)$.
%
%======================================================================================
   \begin{figure}[h!]
      \centering
      {\includegraphics[angle=0, width = 0.6\textwidth]{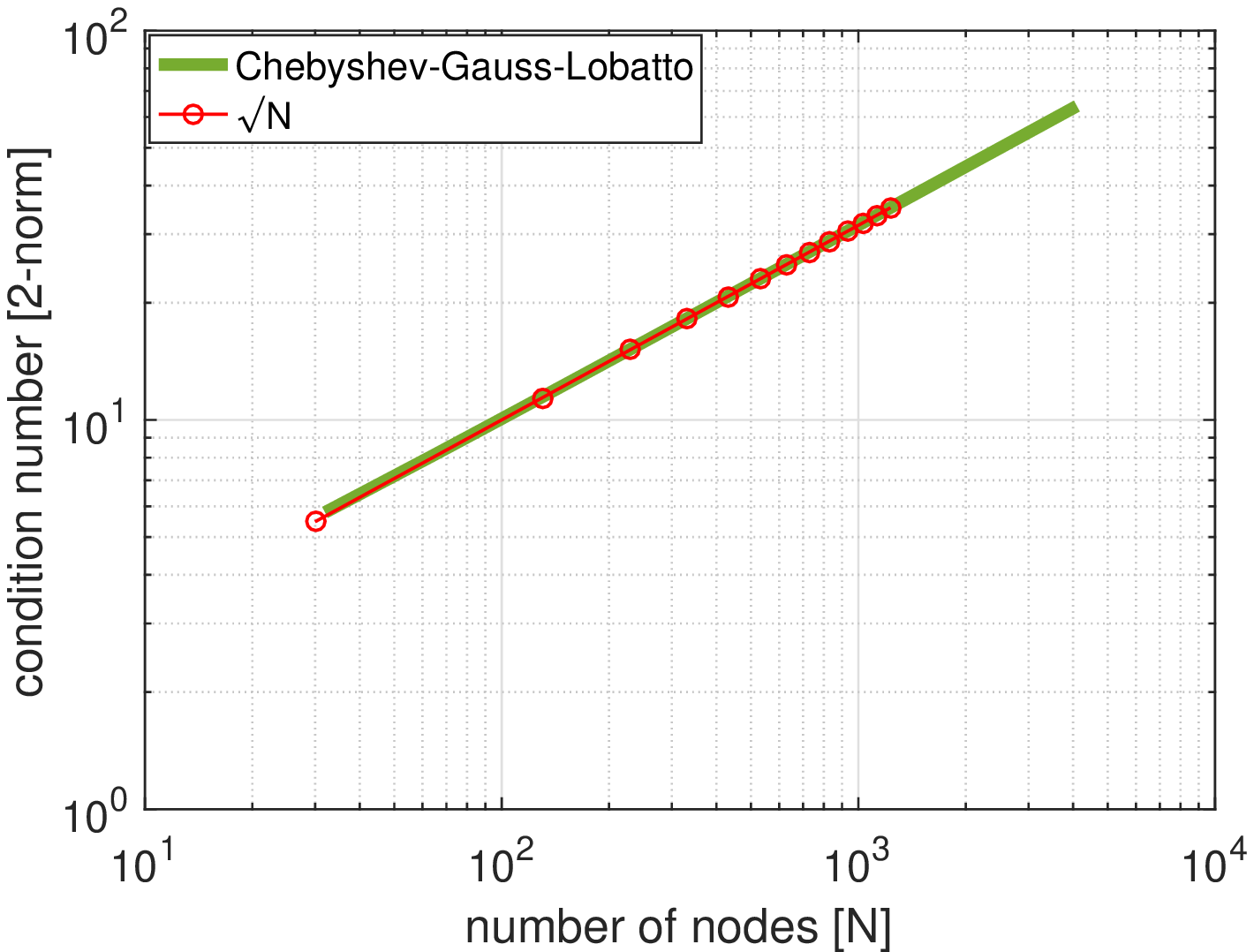}}
      \caption{\textsf{Demonstrating $\mathcal{O}\left(\sqrt{N}\right)$-growth in the condition number of the matrix $\bA^{\text{Birk}}_N$ defined in \eqref{eq:AN-Birk-def} with evaluations over $N$ CGL nodes.}}
    \label{fig:cond_AN}
   \end{figure}
%==========================================================================================
%

The condition numbers for other choices of grid selections (see Remark \ref{remark:Birkhoff-galore}) have similar behavior; see, for example, Fig.~\ref{fig:CondNo}. In this paper we largely focus on Lobatto-based grid selections because they are the correct choice for solving generic finite-horizon optimal control problems (see for example, Ref.~\cite{PSReview-ARC-2012}, Section 4.4.4 in Ross\cite{ross-book}, and Ref.~\cite{advances}). Radau-based mesh points  --- first introduced in \cite{Radau-GNC05} and \cite{Radau-JGCD} for optimal control applications --- are more appropriate for infinite-horizon optimal control problems because the non-unit weight function associated with the Legendre-Radau-grid vanishes at the $+1$ (or $-1$) point\cite{PSReview-ARC-2012,Radau-JGCD}.  Consequently, a Legendre-Radau grid is mathematically justified only for some limited choices of boundary conditions in a finite-horizon optimal control problem\cite{PSReview-ARC-2012,advances}.  \emph{\textbf{Using a Radau mesh for finite-horizon optimal control applications with arbitrary boundary conditions produces well-known convergence issues and control-chatter problems\cite{PSReview-ARC-2012,ross-book,advances}}}.  Pure Gauss points are even more limiting with regards to their suitability for solving generic optimal control problems\cite{PSReview-ARC-2012,advances}.

\section{An Illustrative Example}
There are many emerging problems in aerospace trajectory optimization that require a fine mesh\cite{ross-challenge,nolcos16,napa16,conway:book-chapter}. A detailed discussion of these problems is beyond the scope of this paper; hence, we use a proxy low-thrust orbital maneuvering problem\cite{bellman-low-t, conway:book-chapter} as a prototypical challenge to illustrate the numerics of the new PS methods developed in the preceding sections.  That is, we choose this problem for illustrative purposes only and not necessarily to address the myriad of issues that are associated with the specifics of low-thrust orbital maneuvering.  To this end, consider a minimum-time circle-to-circle continuous thrust orbit transfer problem (see Fig.~\ref{fig:OXfer}) given by,
%
%-----------------------------------------------------------------------------
\begin{eqnarray*}
& \bx = (r, \theta, v_r, v_t) \in \real 4, \quad \bu = \alpha \in \Real & \nonumber \\ \\
(O_\text{Xfer}) &\left\{
\begin{array}{lrl}
\textsf{Minimize } & J[\bx(\cdot), \bu(\cdot), t_f] =& t_f
\\ \\
\textsf{Subject to}
& \dot r =& v_r \\ \\
& \dot\theta =& \displaystyle\frac{v_t}{r} \\ \\
&\dot v_r =& \displaystyle\frac{v_t^2}{r}-\frac{\mu}{r^2}+A\sin\alpha  \\ \\
&\dot v_t =& -\displaystyle\frac{v_r v_t}{r}+ A\cos\alpha \\
%&\dot m =& - T/{v_e}\\
&t_0 = & 0 \\
&\left(r_0, \theta_0, v_{r_0}, v_{t_0}\right) =& \left(r^0, 0, 0, \sqrt{\mu/r^0}\right)\\
%&e_1(\bx_f) := x_f^2 + y_f^2 -R^2 = & 0 \\
%&e_2(\bx_f) := v_{x_f}^2 + v_{y_f}^2 -\frac{\mu}{R} = & 0
&\left(r_f, v_{r_f}, v_{t_f}\right) =& \left(r^f, 0, \sqrt{\mu/r^f}\right)
\end{array} \right. &
\end{eqnarray*}
%------------------------------------------------------------------------------
%
%
%
%======================================================================================
   \begin{figure}[h!]
      \centering
      {\includegraphics[angle=0, height=3in, width = 0.6\textwidth]{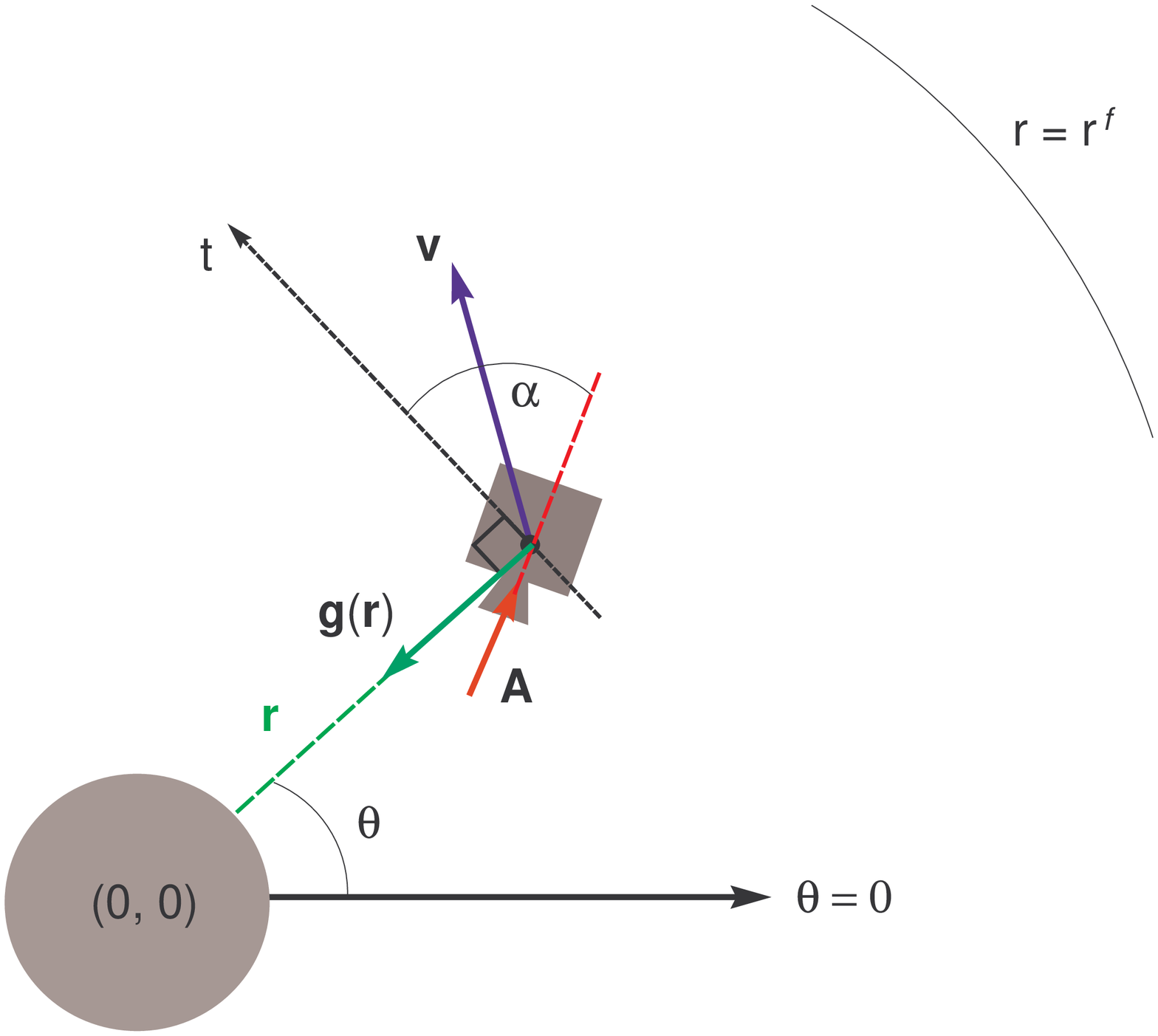}}
      \caption{\textsf{Schematic for Problem $O_\textrm{Xfer}$.}}
    \label{fig:OXfer}
   \end{figure}
%==========================================================================================
%
where, $r$ is the radial position of the spacecraft, $\theta$ is the true anomaly, $v_r$ and $v_t$ are the radial and transverse components of the spacecraft velocity, $\alpha$ is the thrust steering angle, $\mu$ is the gravitational parameter and $A$ is the magnitude of the constant thrust acceleration. What makes Problem $O_\text{Xfer}$ challenging is that if $A$ is low and $\left(r^f/r^0\right)$ is high, then the number of orbital revolutions to attain the final orbit is high; hence, from an elementary application of the Nyquist-Shannon sampling theorem, it is apparent that the number of nodes must be high.  \emph{\textbf{In principle, solving a problem with a large number of nodes is not a significant technological hurdle\cite{ross-challenge,ross-book}.  As an elementary illustration of the technology requirements, we note that a million variables only requires 8~MB of computer memory.}}  What makes a high-node problem hard to solve is science and not technology; that is, the problem is the high condition number of the differentiation matrix associated with a PS method. For example, from Fig.~\ref{fig:CondNo}, the condition number of a PS differentiation matrix for $N= 1000$ is about $10^6$; i.e., $\mathcal{O}(N^2)$.  Because it is equivalent to a right-preconditioned Lagrange PS method (see \eqref{eq:rightProbPNa}), a Birkhoff PS method mitigates the problem of high condition numbers.  For example, for $N=1000$, we expect the condition number to be about $30$; see Fig.~\ref{fig:cond_AN}.  To illustrate these features, we solve Problem $O_\text{Xfer}$ using the following canonical units:
\begin{equation}
\begin{aligned}
\text{Distance Unit} &= r^0\\
\text{Speed Unit} & = \sqrt{\mu/r^0}\\
\text{Time Unit} & = \sqrt{\left(r^0\right)^3/\mu}\\
\text{Acceleration Unit} & = \mu/\left(r^0\right)^2
\end{aligned}
\end{equation}
The boundary conditions for a representative LEO-to-GEO orbit transfer problem (in canonical units) is given by,
\begin{equation}\label{eq:leo2geoBC}
\begin{aligned}
\left(r_0, \theta_0, v_{r_0}, v_{t_0}\right) =& \left(1, 0, 0, 1\right)\\
\left(r_f, v_{r_f}, v_{t_f}\right) =& \left(6, 0, \sqrt{1/6}\right)
\end{aligned}
\end{equation}
A low-thrust solution ($A = 5 \times 10^{-4}$) resulting from an application of Birkhoff PS method is shown in Fig.~\ref{fig:geo1024_polar}.
%
%======================================================================================
   \begin{figure}[h!]
      \centering
      {\includegraphics[angle=0, width=0.6\textwidth]{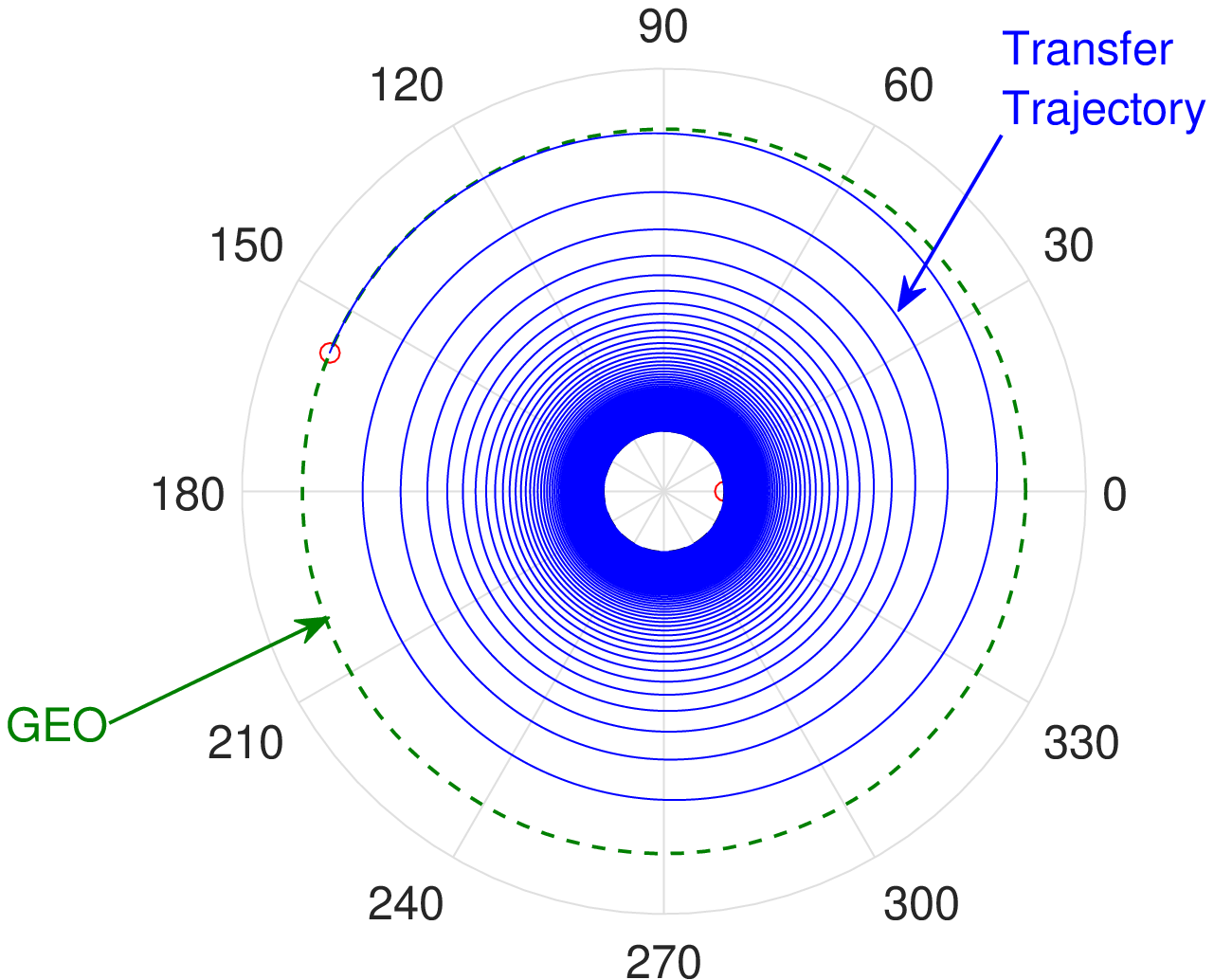}}
      \caption{\textsf{A Birkhoff PS solution for Problem $O_\text{Xfer}$ and boundary conditions given by \eqref{eq:leo2geoBC}.}}
    \label{fig:geo1024_polar}
   \end{figure}
%==========================================================================================
%
This solution was obtained using a CGL mesh and the right preconditioning method defined by \eqref{eq:rightProbPNa}.  Results for other cases and problems including solutions from the vast variety of other possible Birkhoff PS methods outlined in Remark \ref{remark:Birkhoff-galore} are discussed in \cite{koeppen}.  In order to limit the scope of this section, we only discuss this particular case.

The time to accomplish the maneuver shown in Fig.~\ref{fig:geo1024_polar} was approximately 1187 time units; this translates to approximately 13 days of continuous low-thrust maneuvering.  The number of CGL points was $N = 1024$.  In other words, $r(t)$ shown in Fig.~\ref{fig:geo1024_polar} is given by a polynomial of over a thousand degree!  The ``discrete'' steering angle at 1024 CGL points is shown in Fig.~\ref{fig:geo1024_alpha}.
%
%======================================================================================
   \begin{figure}[h!]
      \centering
      {\includegraphics[angle=0, width = 0.6\textwidth]{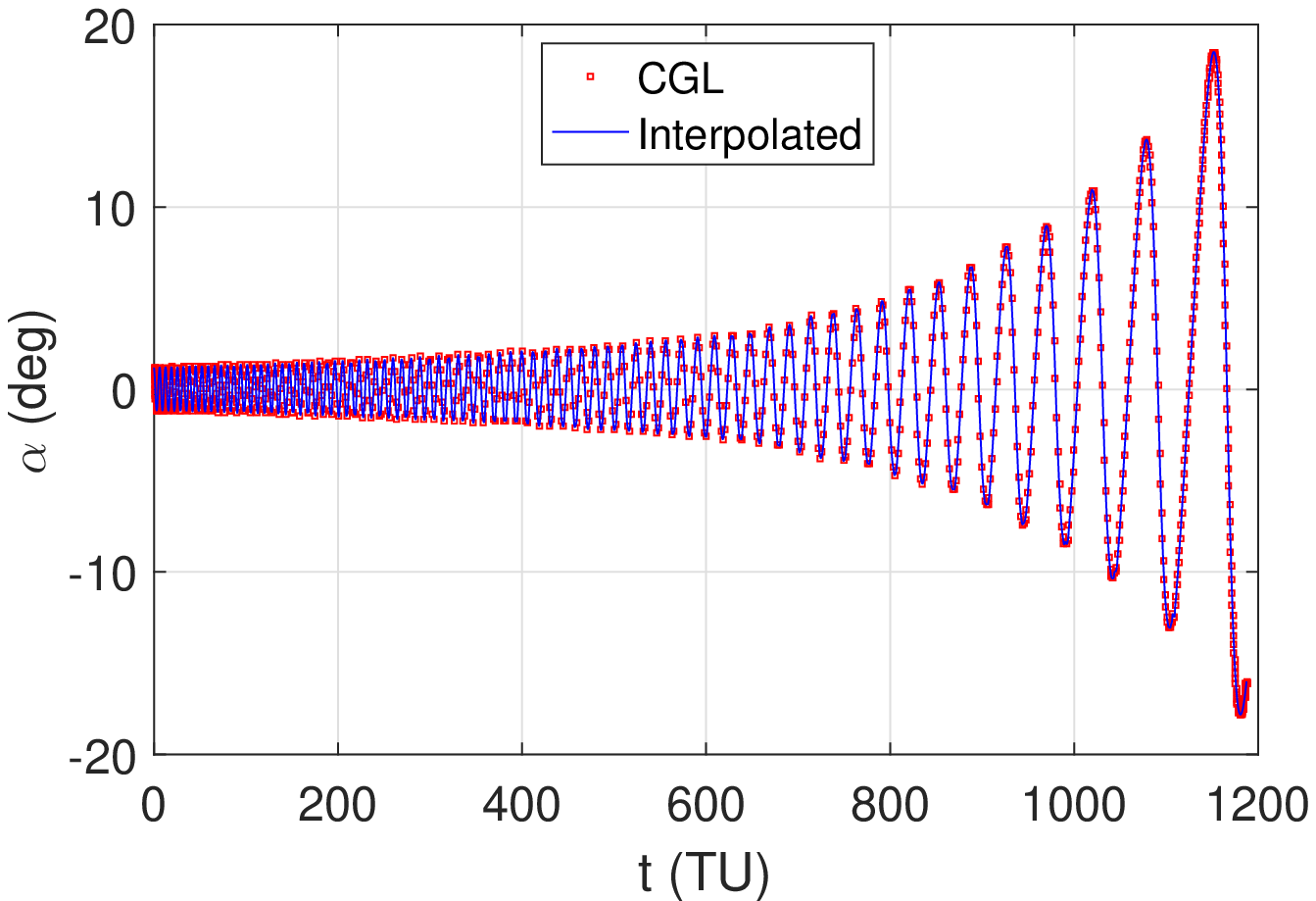}}
      \caption{\textsf{A CGL-interpolated PS control solution $t \mapsto \alpha(t)$ corresponding to the plot shown in Fig.~\ref{fig:geo1024_polar}.}}
    \label{fig:geo1024_alpha}
   \end{figure}
%==========================================================================================
%
When this steering angle is interpolated through the CGL mesh, we get the continuous function $t \mapsto \alpha$ indicated by the solid line in Fig.~\ref{fig:geo1024_alpha}. Using the interpolated function $t \mapsto \alpha$, the dynamical equations can be easily integrated using ode45 (from MATLAB) to generate a propagated trajectory.  This propagated trajectory is included in Fig.~\ref{fig:geo1024_polar}; however, this plot is not noticeable  in this figure because the differences are too small to be visually apparent.
%
%======================================================================================
   \begin{figure}[h!]
      \centering
      {\includegraphics[angle=0, width = 0.6\textwidth]{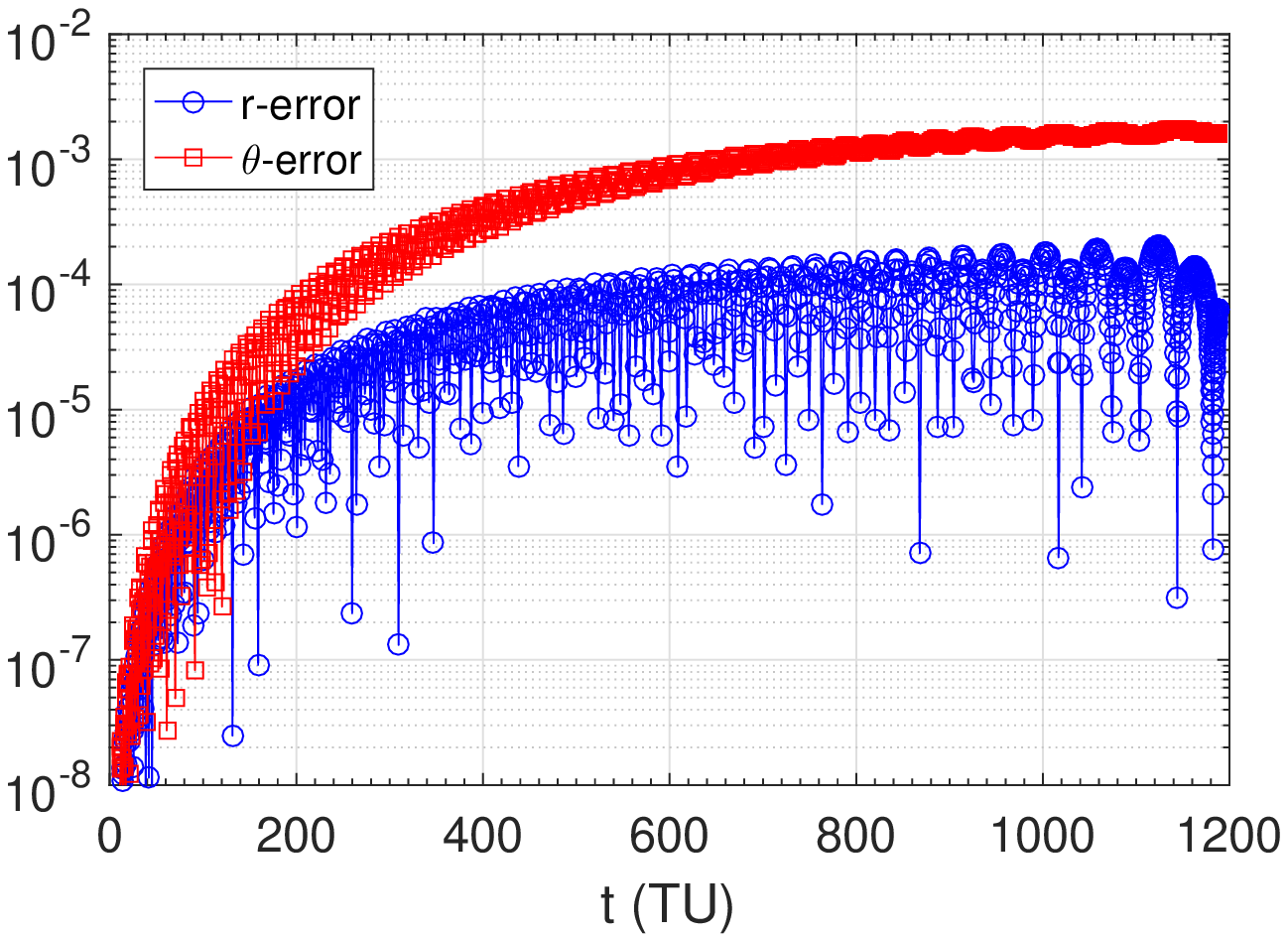}}
      \caption{\textsf{Error plots as measured by the difference between a Birkhoff PS solution and a Runge-Kutta propagation through the CGL-interpolated solution shown in Fig.~\ref{fig:geo1024_alpha}.}}
    \label{fig:geo1024_error}
   \end{figure}
%==========================================================================================
%
Consequently, we provide an error plot in Fig.~\ref{fig:geo1024_error} as a matter of completeness.

\section{Conclusions}
State trajectories in aerospace dynamical systems are absolutely continuous; hence, by an elementary application of the Stone-Weierstrass theorem, they can be expressed to any arbitrary precision in terms of some finite-order polynomial. In using the big two polynomials, namely Legendre and Chebyshev, pseudospectral (PS) optimal control theory has found numerous applications that include include flight implementations in aerospace engineering. In many practical PS optimal control applications, very accurate solutions can be generated using polynomials of order $N \le 100$, thanks to the spectral (i.e., near exponential) rate of convergence.  Emerging applications have sought to generate solutions in excess of $N > 1000$.  The challenge with such high orders is the high condition number \big($\sim\mathcal{O}\left(N^2\right) $ \big) of the associated differentiation matrix.  This is largely a science problem rather than a technological one.  That is, a computer's processing capability is not a limiting factor because \emph{\textbf{8MB of memory for a million variables is no longer a technological hurdle}}. New advancements in the science associated with Birkhoff interpolation provide a dramatic drop in the condition number: from $\mathcal{O}\left(N^2\right) $ to $\mathcal{O}\big(\sqrt{N}\big) $ or even $\mathcal{O}\big(1\big)$ in some cases.  In advancing these new ideas, the problem with condition numbers is now relegated to specific applications and not to PS optimal control techniques. Nonetheless, a significant amount of new research on Birkhoff PS methods remains to be done.

\section*{Acknowledgment}
We take this opportunity to thank the anonymous reviewers for providing many detailed comments.  Their comments greatly improved the quality of this paper.

%%%%%%%%
%Bibliography
%%%%%%%%

%\begin{spacing}{1.0}

%\end{spacing}

\end{document}